\documentclass[english,notitlepage]{amsart}
\usepackage[latin9]{inputenc}
\usepackage{geometry}
\usepackage{esint}
\usepackage{hyperref}
\usepackage{amssymb,amsmath,amsthm,amscd}

\makeatletter
\numberwithin{equation}{section}
\numberwithin{figure}{section}

\theoremstyle{plain}
\newtheorem{thm}{Theorem}
  \theoremstyle{plain}
  \numberwithin{thm}{section}
  \newtheorem{cor}[thm]{Corollary}
  \theoremstyle{plain}
  \newtheorem{lem}[thm]{Lemma}
  \theoremstyle{remark}
  \newtheorem{rem}[thm]{Remark}
    \theoremstyle{remark}

   \theoremstyle{plain}
  
  \def\Ddots{\mathinner{\mkern1mu\raise\p@
\vbox{\kern7\p@\hbox{.}}\mkern2mu
\raise4\p@\hbox{.}\mkern2mu\raise7\p@\hbox{.}\mkern1mu}}
\makeatother

\usepackage{babel}

\numberwithin{definition}{section}

\newtheorem{result}{Result}

\newcommand{\norm}[1]{\left\| #1 \right\|}

\newcommand{\eklm}[1]{\left\langle #1 \right\rangle}

\renewcommand{\d}{\,d}
\newcommand{\N}{{\mathbb N}}
\newcommand{\Z}{{\mathbb Z}}
\newcommand{\C}{{\mathbb C}}

\newcommand{\R}{{\mathbb R}}

\newcommand{\B}{{\mathcal B}}

\newcommand{\1}{{\bf 1}}
\renewcommand{\epsilon}{\varepsilon}

\renewcommand{\rho}{\varrho}

\newcommand{\bdm}{\begin{displaymath}}
\newcommand{\edm}{\end{displaymath}}
\newcommand{\bq}{\begin{equation}}
\newcommand{\eq}{\end{equation}}
\newcommand{\bqn}{\begin{equation*}}
\newcommand{\eqn}{\end{equation*}}

\newcommand{\Cinft}{{\rm C^\infty}}
\newcommand{\CT}{{\rm C^\infty_c}}

\newcommand{\Sob}{{\rm H}}
\renewcommand{\L}{{\rm L}}

\newcommand{\vol}{\text{vol}\,}

\DeclareMathOperator{\supp}{supp\,}

\DeclareMathOperator{\tr}{tr}
\DeclareMathOperator{\gd}{\partial}

\begin{document}
\author{Benjamin K\"uster}
\address{Philipps-Universit\"at Marburg, Fachbereich Mathematik und Informatik, Hans-Meerwein-Str., 35032 Marburg, Germany}
\email{\href{mailto:bkuester@mathematik.uni-marburg.de}{bkuester@mathematik.uni-marburg.de}}
\title{Semiclassical functional calculus for $h$-dependent functions}
\date{\today}

\begin{abstract}We study the functional calculus for operators of the form $f_h(P(h))$ within the theory of semiclassical pseudodifferential operators, where  $\{f_h\}_{h\in (0,1]}\subset \CT(\R)$ denotes a family of $h$-dependent functions satisfying some regularity conditions, and $P(h)$ is either an appropriate self-adjoint semiclassical pseudodifferential operator in $\L^2(\R^n)$ or a Schr\"odinger operator in $\L^2(M)$, $M$ being a closed Riemannian manifold of dimension $n$. The main result is an explicit semiclassical trace formula with remainder estimate that is well-suited for studying the spectrum of $P(h)$ in spectral windows of width of order $h^\delta$, where $0\leq \delta <\frac{1}{2}$.
\end{abstract}

\maketitle

\setcounter{tocdepth}{1}
\tableofcontents{}

\section{Introduction}
\subsection{Motivation and problem}\label{sec:motivgoalsI}
It is very well known that the functional calculus for unbounded self-adjoint operators combines with the symbolic calculus for semiclassical pseudodifferential operators in a useful way, a major class of examples being given by trace formulas, see \cite{dimassi-sjoestrand,helffer-robert83} for the $\R^n$ case and \cite{zworski} for the closed manifold case. However, while applying the functional calculus for semiclassical Schr\"odinger operators on a closed Riemannian manifold in the context of equivariant quantum ergodicity \cite{kuester-ramacher15a}, we noticed that in the literature the functions considered in the functional calculus are either supposed to be independent of the semiclassical parameter $h$, or the $h$-dependence is (implicitly) allowed but the results are not stated in an explicit enough form. Therefore, the main purpose of this paper is to prove explicit trace formulas for the functional calculus of Schr\"odinger operators on closed Riemannian manifolds with respect to $h$-dependent functions. 
To explain things more precisely, let $X$ be either the euclidean space $\R^n$ or a closed connected Riemannian manifold of dimension $n$, endowed with the Riemannian measure $dX$, and $P(h)$ be the self-adjoint extension of an essentially self-adjoint semiclassical pseudodifferential operator in $\L^2(X)$ with real-valued semiclassical principal symbol $p$. Given a function $f\in \CT(\R)$, and assuming that $p$ satisfies reasonable technical ellipticity conditions, $f(P(h))$ extends to a bounded semiclassical pseudodifferential operator on $\L^2(X)$ with semiclassical principal symbol $f\circ p$, see \cite{dimassi-sjoestrand,zworski}. Provided that one has chosen $f$ and $p$ appropriately such that $f(P(h))$ is of trace class, one obtains the asymptotic trace formula
\bq
(2\pi h)^n\tr_{\L^2(X)}f(P(h))=\intop_{T^*X} f\circ p \,\d (T^*X) + \mathrm{O}(h)\qquad \text{as }h\to 0, \label{eq:traceform00}
\eq
where $d(T^*X)$ is the volume form defined by the canonical symplectic form on the co-tangent bundle $T^*X$. Now, suppose we are in the situation that there are numbers $E\in \R$ and $c,\varepsilon>0$ such that for each $h\in (0,1]$ the following holds: $p^{-1}([E-\varepsilon,E+c+\varepsilon])$ is compact, $E$ and $E+c$ are regular values of $p$, and $P(h)$ has discrete spectrum in $[E-\varepsilon,E+c+\varepsilon]$ consisting of only finitely many eigenvalues $\{E_j(h)\}_{j\in J(h)\subset \N}$. Then, (\ref{eq:traceform00}) can be applied for each $f\in \CT([E-\varepsilon/2,E+c+\varepsilon/2])$, and by approximating the characteristic function of $[E,E+c]$ with such functions, one gets the semiclassical Weyl law 
\bq
(2\pi h)^n \# \big\{j\in J(h): E_j(h)\in [E,E+c]\big\}=\vol_{T^*X}\big(p^{-1}([E,E+c])\big)+\mathrm{o}(1)\label{eq:weyllaw1}
\eq
as $h\to 0$, see \cite[Cor.\ 9.7]{dimassi-sjoestrand} and \cite[Thm.\ 14.11]{zworski}. In general, it is more desirable to study spectral windows of shrinking width of the form $$[E,E+c(h)],\qquad \lim_{h\to 0}c(h)=0,$$
since in this case the leading term in Weyl-type formulas as above turns into an integral over the compact hypersurface $p^{-1}(\{E\})\subset T^*X$ with respect to the induced Liouville measure, and one may drop some technical hypotheses which would be necessary without the localization to such a hypersurface. However, this is not possible in the functional calculus approach sketched above, where the function $f$ is fixed and independent of $h$. \\

In this paper, we shall study in detail the semiclassical functional calculus for operators of the form $f_h(P(h))$ within the theory of semiclassical pseudodifferential operators, where $f_h\in \CT(\R)$ is explicitly allowed to depend on $h\in(0,1]$. \\

The largest class of $h$-dependent functions that we will consider is given by $\bigcup_{\delta\in[0,\frac{1}{2})}\mathcal{S}^{\mathrm{comp}}_{\delta}$, where for each $\delta\in[0,\frac{1}{2})$ the symbol $\mathcal{S}^{\mathrm{comp}}_{\delta}$ denotes the set of all families $\{f_h\}_{h\in (0,1]}\subset \CT(\R)$ such that 
\begin{itemize}
\item[(1)] $\{f_h\}_{h\in (0,1]}$ defines an element of the semiclassical symbol class $S_{\delta}(1_\R)$, meaning that 
\[
\norm{f_h^{(j)}}_\infty=\mathrm{O}(h^{-\delta j})\qquad\text{as }h\to 0,\quad j=0,1,2,\ldots, 
\]
where $f_h^{(j)}$ denotes the $j$-th derivative of $f_h$;
\item[(2)] the diameter of the support of $f_h$ does not grow faster than polynomially in $h^{-1}$ as $h\to 0$.
\end{itemize}
The second property means that there is some $N\geq0$ such that $\mathrm{diam}(\supp f_h)=\mathrm{O}(h^{-N})$ as $h\to 0$. This is a very mild technical condition; in usual applications the diameter of the support of $f_h$ will be bounded or even tend to zero as $h\to 0$. It is therefore convenient to introduce also the subset of $\mathcal{S}^{\mathrm{comp}}_{\delta}$ given by 
\bqn
\mathcal{S}^{\mathrm{bcomp}}_{\delta}:=\big\{\{f_h\}_{h\in (0,1]}\in\mathcal{S}^{\mathrm{comp}}_{\delta}: \exists \text{ compact interval }I\subset \R\text{ with }\supp f_h\subset I\;\,\forall\;h\in (0,1]\big\}.
\eqn
The class $\mathcal{S}^{\mathrm{bcomp}}_{\delta}$ is technically easier to handle and we stress again that the loss of generality from $\mathcal{S}^{\mathrm{comp}}_{\delta}$ to $\mathcal{S}^{\mathrm{bcomp}}_{\delta}$ seems to be irrelevant to most applications. In the following, we will use a shorter notation and just write $f_h\in \mathcal{S}^{\mathrm{bcomp}}_{\delta}$ or $f_h\in \mathcal{S}^{\mathrm{comp}}_{\delta}$.
\subsection{Goals and methods}
We pursue two main goals. The first is to prove that $f_h(P(h))$ is a semiclassical pseudodifferential operator, provided that $f_h\in\mathcal{S}^{\mathrm{comp}}_{\delta}$ and $P(h)$ is the self-adjoint extension of an appropriate essentially self-adjoint semiclassical pseudodifferential operator in $\R^n$. This involves relating the abstract functional calculus to the semiclassical symbolic calculus with suitable estimates. Operator norm and trace norm estimates have been carried out very precisely for appropriate classes of operators in \cite{helffer-robert83} and \cite{dimassi-sjoestrand}, so that, thanks to the strong existing literature on the $\R^n$ case, it is not hard to achieve our first goal. We mainly work out certain critical steps in existing proofs in a more explicit form so that they generalize to the classes $\mathcal{S}^{\mathrm{comp}}_{\delta}$. The main issues here are estimates that are uniform in $h$ for fixed $f$, but which depend on $f$ and hence are no longer uniform in $h$ when $f=f_h$. Our second goal is then to provide a detailed treatise of the semiclassical functional calculus for Schr\"odinger operators on a closed connected $n$-dimensional Riemannian manifold $M$ of the form\bq
P(h)=-h^2\Delta +V,\qquad V\in \Cinft(M,\R),\qquad P(h):\Sob^2(M)\to \L^2(M),\label{eq:schrop}
\eq
where $V$ is a real-valued potential, $\Delta$ is the unique self-adjoint extension of the Laplace-Beltrami operator $\breve \Delta: \Cinft(M)\to \Cinft(M)\subset \L^2(M)$,  and $\Sob^2(M)$ denotes the second Sobolev space. Thus, $P(h)$ is the unique self-adjoint extension of the essentially self-adjoint operator
\bqn
\breve{P}(h):=-h^2\breve \Delta + V,\qquad \breve{P}(h):\Cinft(M)\to \Cinft(M)\subset \L^2(M).
\eqn
As is well known,  the spectrum of $P(h)$ is discrete for each $h\in(0,1]$ and accumulates only at $+\infty$, see \cite[Chapter 14]{zworski}. We write $$p(x,\xi):=\norm{\xi}^2_x+V(x),\qquad p:T^*M\to \R,$$ for the Hamiltonian function associated to $P(h)$, which represents its semiclassical principal symbol. Apart from establishing that $f_h(P(h))$ is a semiclassical pseudodifferential operator when $f_h\in \mathcal{S}^{\mathrm{comp}}_{\delta}$, we are interested in an explicit expression for $f_h(P(h))$ in terms of local quantizations of symbol functions in $\R^n$, which can be used flexibly to prove new semiclassical trace formulas that are well-suited for studying spectral windows of width of order $h^\delta$, where $0\leq\delta<\frac{1}{2}$. This involves relating the abstract functional calculus to the local semiclassical symbolic calculus with trace norm remainder estimates, which we do under the assumption that $f_h\in \mathcal{S}^{\mathrm{bcomp}}_{\delta}$. In contrast to the euclidean case, there seems to be no reference in the literature where for a Schr\"odinger operator $P(h)$ on a closed manifold the transition from the global operator $f_h(P(h))$ to the locally defined quantizations, obtained by introducing an atlas and a partition of unity, is made in a way such that the \emph{trace norm} of the remainder operators is precisely controlled, not even if $f_h$ is actually independent of $h$. 
\subsection{Summary of main results} In what follows, the results of this paper are presented in a slightly condensed form. In the first two results, which are preliminaries for the third result, we consider semiclassical pseudodifferential operators in $\L^2(\R^n)$, with the notation $\overline{A}$ for the self-adjoint extension of an essentially self-adjoint operator $A$, and $\mathrm{Op}_h$ for the Weyl quantization. For the other required definitions, in particular those of \emph{order functions} and the associated notion of ellipticity, the symbol classes $S^k_{\delta}(\mathfrak{m})$ and $S^m_{\delta}(M)$, and the operator classes $\Psi^m_{\delta} (M)$, we refer the reader to Section \ref{subsec:semiclassical}. To state our results, fix a number $\delta\in [0,\frac{1}{2})$. The main result of Section \ref{sec:resultsrn} is

\begin{result}[{Theorem \ref{thm:semiclassfunccalcrn},  $h$-dependent version of \cite[Theorem 8.7]{dimassi-sjoestrand}}]\label{res:1} 
Let  $\mathfrak{m}:\R^{2n}\to (0,\infty)$ be an order function with $\mathfrak{m}\geq 1$, and let $s\in S(\mathfrak{m})$ be a real-valued symbol function such that $s+i$ is $\mathfrak{m}$-elliptic. Choose $f_h\in \mathcal{S}^{\mathrm{comp}}_{\delta}$. Then, for small $h$ the operator $f_h\big(\overline{\mathrm{Op}_h(s)}\big):\L^2(\R^n)\to \L^2(\R^n)$ is a semiclassical pseudodifferential operator. More precisely, there is a symbol function $a\in \bigcap_{k\in \N}S_{\delta}(\mathfrak{m}^{-k})$ and a number $h_0\in (0,1]$ such that for $h\in (0,h_0]$
\[
f_h\Big(\overline{\mathrm{Op}_h(s)}\Big)=\mathrm{Op}_h(a).
\]
Moreover, if $s$ has an asymptotic expansion in $S(\mathfrak{m})$ of the form $
s\sim \sum_{j=0}^\infty h^j s_j$, then $a$ has an expansion in $S_{\delta}(1/\mathfrak{m})$, with explicitly known coefficients, of the form
\bq
a\sim \sum_{j=0}^\infty a_j,\qquad a_j\in S_{\delta}^{j(2\delta-1)}(1/\mathfrak{m}),\qquad a_0(y,\eta,h)=f_h(s_0(y,\eta,h)).\label{eq:expansion333333}
\eq
\end{result}
The second result concerns the Schr\"odinger operator (\ref{eq:schrop}) on a closed connected Riemannian manifold $M$ of dimension $n$. We obtain in Section \ref{sec:closedriem}:
\begin{result}[{Theorem \ref{thm:semiclassfunccalcmfld}, $h$-dependent version of \cite[Theorem 14.9]{zworski}}]\label{res:partIresult2}Choose a function $\rho_h\in \mathcal{S}^{\mathrm{comp}}_{\delta}$. Then, for small $h$ the operator $\rho_h(P(h))$ is a semiclassical pseudodifferential operator on $M$ of order $(-\infty,\delta)$, and its principal symbol is represented by the function $\rho_h\circ p$.
\end{result}
In order to prove trace formulas for a semiclassical pseudodifferential operator on a manifold, a standard approach is to approximate the operator by pullbacks of semiclassical pseudodifferential operators in $\R^n$ by introducing an atlas and a partition of unity, up to a trace class remainder operator with small trace norm. In addition, one would like to localize the leading term in the obtained trace formulas using an operator $B\in \Psi^0_{\delta}(M)$ with principal symbol represented by a symbol function $b\in S^0_{\delta}(M)$. Thus, let us introduce a finite atlas $$\{U_\alpha,\gamma_\alpha\}_{\alpha \in \mathcal{A}},\qquad \gamma_\alpha: U_\alpha\stackrel{\simeq}{\to} \R^n,\qquad U_\alpha\subset M\;\text{open}.$$
We choose the whole euclidean space $\R^n$ as the image of our charts in order to avoid problems related to the fact that pseudodifferential operators are non-local. Furthermore, in order to state our next result we require a smooth partition of unity $\{\varphi_\alpha\}_{\alpha \in \mathcal{A}}$ on $M$  subordinate to $\{U_\alpha\}_{\alpha \in \mathcal{A}}$ and for each $\alpha\in\mathcal{A}$ an associated triple of cutoff functions $\overline \varphi_\alpha,\overline{\overline \varphi}_\alpha,\overline{\overline {\overline \varphi}}_\alpha\in \CT(U_\alpha)$ with $$\overline \varphi_\alpha\equiv 1\text{ on }\supp  \varphi_\alpha,\qquad \overline{\overline \varphi}_\alpha\equiv1\text{ on }\supp  \overline \varphi_\alpha,\qquad \overline{\overline {\overline \varphi}}_\alpha\equiv 1\text{ on }\supp\overline {\overline \varphi}_\alpha.$$  For each chart, define a local symbol function
\bqn
u_{\alpha,0}(y,\eta,h):=\big((\rho_h\circ p)\cdot b\big)\big(\gamma_\alpha^{-1}(y),(\partial\gamma_\alpha^{-1})^T\eta,h\big)\cdot\varphi_\alpha\big(\gamma_\alpha^{-1}(y)\big)
\eqn
where $(y,\eta)\in\R^{2n},\; \alpha\in\mathcal{A},\; h\in (0,1]$. Then, one has the following
\begin{result}[Theorem \ref{thm:localtraceform}]\label{res:3}Suppose that $\rho_h\in \mathcal{S}^{\mathrm{bcomp}}_{\delta}$. Then, for each $N\in \N$, there is a number $h_0\in(0,1]$, a collection of symbol functions $\{r_{\alpha,\beta,N}\}_{\alpha,\beta\in \mathcal{A}}\subset S^{2\delta-1}_{\delta}(1_{\R^{2n}})$ and an operator $\mathfrak{R}_{N}(h)\in \B(\L^2(M))$ such that
\begin{itemize}
\item one has for all $f\in \L^2(M)$, $h\in (0,h_0]$ the relation \begin{multline}
B\circ \rho_h(P(h))(f)=\sum_{\alpha \in \mathcal{A}}\overline \varphi_\alpha\cdot\mathrm{Op}_h(u_{\alpha,0})\big((f\cdot \overline{\overline {\overline \varphi}}_\alpha)\circ \gamma_\alpha^{-1}\big)\circ\gamma_\alpha\\
+\sum_{\alpha,\beta \in \mathcal{A}}\overline \varphi_\beta\cdot\mathrm{Op}_h(r_{\alpha,\beta,N})\big((f\cdot \overline{\overline \varphi}_\alpha\cdot\overline {\overline {\overline \varphi}}_\beta)\circ\gamma_\beta^{-1}\big)\circ\gamma_\beta  \;+\; \mathfrak{R}_{N}(h)(f);\label{eq:result1formula1214}
\end{multline}

\item the operator $\mathfrak{R}_N(h)\in \B(\L^2(M))$ is of trace class and its trace norm fulfills
\bqn
\norm{\mathfrak{R}_N(h)}_{\mathrm{tr},\L^2(M)}=\mathrm{O}\big(h^{N}\big)\quad\text{as }h\to 0;
\eqn
\item for fixed $h\in (0,h_0]$, each symbol function $r_{\alpha,\beta,N}$ is an element of $\CT(\R^{2n})$  that fulfills 
\bqn
\supp  r_{\alpha,\beta,N}\subset \supp   \big((\rho_h\circ p)\cdot b \cdot\varphi_\alpha\big)\circ(\gamma_\alpha^{-1},(\partial\gamma_\alpha^{-1})^T).
\eqn
\end{itemize}
\end{result}

\subsection{Applications}
In general,  Result \ref{res:3} can be used to prove an asymptotic semiclassical trace formula with non-trivial remainder estimates for an operator of the form
\[
T\circ B\circ \rho_h(P(h))
\]
where $T:\L^2(M)\to \L^2(M)$ is some bounded, explicitly known operator. For example, $T$ can be defined using an additional structure on the manifold $M$. Then, by Result \ref{res:3}, one has for each $N\in \N$
\[
\mathrm{tr}_{\L^2(M)}\big[T\circ B\circ \rho_h(P(h))\big]=\mathrm{tr}_{\L^2(M)}\big(T\circ L_N\big) + \mathrm{O}\big(h^N\big)\qquad \text{as }h\to 0,
\]
where $L_N$ is the operator defined by the right hand side of (\ref{eq:result1formula1214}) without $\mathfrak{R}_N(h)(f)$. The significance of Result \ref{res:3} is that proving a trace formula for $T\circ B\circ \rho_h(P(h))$ with remainder of order $h^N$ immediately reduces to calculating the leading term $\mathrm{tr}_{\L^2(M)}(T\circ L_N)$, and this term involves only pullbacks of semiclassical pseudodifferential operators in $\R^n$, so that in the calculations one can rely on the precise symbolic calculus on $\R^n$ and needs to deal only with compactly supported symbol functions.

In the simplest case where $T=B=\1_{\L^2(M)}$, Corollary \ref{thm:semiclassicaltraceformula} yields the $h$-dependent analogue of (\ref{eq:traceform00}) given by
\bqn
(2\pi h)^n\,\mathrm{tr}_{\L^2(M)}\rho_h\left(P(h)\right)=\intop_{T^*M}\rho_h\circ p\,\d(T^*M)+\mathrm{O}\Big(h^{1-2\delta}\,\mathrm{vol}_{\,T^*M}\big(\supp \rho_h\circ p\big)\Big)\quad \text{as }h\to 0.
\eqn
Provided that $\delta<\frac{1}{3}$, this leads directly to an improved version of (\ref{eq:weyllaw1}) given by
\bq
(2\pi)^n h^{n-\delta} \# \big\{j\in J(h): E_j(h)\in [E,E+h^\delta]\big\}=\vol \,p^{-1}(\{E\})+\mathrm{O}\big(h^\delta+h^{\frac{1}{3}-\delta}\big),\label{eq:weyllaw2}
\eq
where the volume is now measured using the induced Liouville measure on $p^{-1}(\{E\})$, compare \cite[proof of Thm.\ 4.1]{kuester-ramacher15a}. Of course, formula (\ref{eq:weyllaw2}) is far from optimal in terms of its quantitative statement (see Subsection \ref{subsec:prevknown1} below), yet it serves as a simple example of the qualitative fact that due to the localization onto the hypersurface $p^{-1}(\{E\})$ it is now enough to assume that $p^{-1}([E-\varepsilon,E+\varepsilon])$ is compact for some small $\varepsilon>0$ and that $E$ is a regular value of $p$, i.e.\ the $c=0$ version of the assumptions required for (\ref{eq:weyllaw1}). 

When choosing $\delta>0$, one can use the operator $B$ to perform a localization to small subsets (for example, single points or geodesics) of $T^*M$ in the semiclassical limit. This  is the small-scale approach, see \cite{han}. By choosing $T$ to be the projection onto a linear subspace $V$ of $\L^2(M)$ and then taking into account the equality $$\mathrm{tr}\big[T\circ B\circ \rho_h(P(h))\big]=\mathrm{tr}\big[B\circ \rho_h(P(h))\circ T\big]=\mathrm{tr}\big[T\circ B\circ \rho_h(P(h))\circ T\big],$$one can use Result $3$ to study the spectral properties of the bi-restriction of $P(h)$ to $V$, still possibly localizing the problem using $B$. 
As a major application, we use Result \ref{res:3} in \cite[Thm.\ 3.1]{kuester-ramacher15a} to prove a singular equivariant semiclassical trace formula for Schr\"odinger operators in case that $M$ carries an isometric effective action of a compact connected Lie group $G$. There, one has $T=T_\chi$, where $T_\chi:\L^2(M)\to \L^2(M)$ is the projection onto an isotypic component of the left-regular $G$-representation in $\L^2(M)$ associated to a character $\chi\in \widehat G$. The calculation of $\mathrm{tr}_{\L^2(M)}(T_\chi\circ L_N)$ reduces to the evaluation of certain oscillatory integrals which can be carried out using a formula from \cite{ramacher10} whose remainder term is of lower order than that in Result \ref{res:3}. Here, knowing a better remainder estimate in Result \ref{res:3} would not improve the results, so in this case Result \ref{res:3} is fully sufficient both qualitatively and quantitatively. The trace formula stated in \cite[Thm.\ 3.1]{kuester-ramacher15a} could not be established without a functional calculus for $h$-dependent functions. It implies a generalized equivariant semiclassical Weyl law with remainder estimate, as well as a symmetry-reduced quantum ergodicity theorem, see \cite{kuester-ramacher15a,kuester-ramacher15b}. 
\subsection{Previously known results}\label{subsec:prevknown1}As mentioned above, an $h$-dependent functional calculus of the form $f_h(P(h))$ has been used in the literature before, but not systematically and usually only implicitly. For example, consider a Schwartz function $\chi:\R\to \R$ whose Fourier transform has compact support. Then a common approach in the literature is to study the operator $\chi\big(\frac{P(h)-E}{h}\big)$ in the context of semiclassical Fourier integral operators, $E\in \R$ being a fixed regular value of $p$. Writing $f^E_h(x):=\chi\big(\frac{x-E}{h}\big)$, this amounts to studying $f^E_h(P(h))$, as in the semiclassical Gutzwiller trace formula  \cite{combescure,gutzwiller} and in the proof of the semiclassical quantum ergodicity theorem in \cite{dyatlov}. Using the same techniques, one can also prove a semiclassical Weyl law for the smallest possible spectral window $[E,E+h]$ with the best possible $\mathrm{O}(h)$-remainder, see \cite{dimassi-sjoestrand,dyatlov,hoermander68}.  However, these techniques are considerably more involved than the simple semiclassical pseudodifferential operator calculus. Another way in which the abstract functional calculus has been used for $h$-dependent functions in the literature is of a very basic form.  Namely, given 
a family $\{f_h\}_{h\in (0,1]}$ of bounded Borel functions on $\R$ with uniformly bounded supremum norms for $h\in(0,1]$, one can use the fact that $f_h(P(h))$ has uniformly bounded operator norm as $h\to 0$, an estimate which follows directly from the spectral theorem. In particular, one considers $f_h(P(h))$ only abstractly as an $h$-dependent bounded operator, and not concretely as a semiclassical pseudodifferential operator or Fourier integral operator. See e.g.\ \cite[Proof of Lemma 3.11]{dyatlov}.
\subsection{Discussion of methods and outlook}\quad Developing the functional calculus for $h$-dependent functions within the theory of semiclassical pseudodifferential operators restricts the applications in spectral analysis to spectral windows of width of order $h^\delta$ with $\delta<\frac{1}{2}$. In particular, the best possible case $\delta=1$ cannot be studied. However, qualitatively there is no significant difference between the cases $\delta=1$ and $\delta>0$, since for any $\delta>0$ the spectral window $[E,E+h^\delta]$ shrinks to a point polynomially fast in the semiclassical limit, leading to a localization on an energy hypersurface in Weyl-type formulas, and if the manifold dimension is greater than $1$, then by Weyl's law the number of eigenvalues in $[E,E+h^\delta]$ grows as $h\to 0$, regardless whether $\delta=1$ or just $\delta>0$. Thus, going beyond the theory of semiclassical pseudodifferential operators could only lead to quantitative improvements, at the expense of losing the simplicity of the symbolic calculus. Although non-optimal, the quantitative results presented here are sufficient for many applications, as outlined above. An obvious possible future line of research consists in studying an explicit functional calculus for $h$-dependent functions within more general semiclassical frameworks of operators, as the theory of semiclassical Fourier integral operators. This will probably yield improved quantitative results.
\subsection{Acknowledgements}
The author would like to thank Maciej Zworski for insightful conversations and Pablo Ramacher for his careful proofreading and many helpful comments.
\section{Preliminaries and Background}\label{sec:prelim}

\subsection{Notation} For two complex vector spaces $V,W$, we write $\mathcal{L}(V,W)$ for the $\C$-linear maps $V\to W$. If $V,W$ are normed vector spaces, we write $\B(V,W)$ for the bounded linear operators $V\to W$ and we set $\B(V):=\B(V,V)$. By a \emph{closed} Riemannian manifold, we mean a compact Riemannian manifold without boundary. 
If $\varphi\in \CT(U)$, where $U$ is an open subset of a smooth manifold $X$, we consider $\varphi$ as a function on $X$ without mentioning the extension by zero explicitly. Similarly, we sometimes consider a function in $\Cinft(X)$ as an element of $\Cinft(T^*X)$ without explicitly mentioning the composition with the cotangent bundle projection. For a chart $\gamma:U\to V$, $V\subset \R^n$, and a function $f\in \Cinft(T^*X)$, we write $f\circ(\gamma^{-1},(\partial\gamma^{-1})^T)$ for the composition of $f$ with $\gamma^{-1}$ in the manifold variable and the adjoint of its derivative in the cotangent space variable. In general, $1_S$ denotes the function with constant value $1$ on a set $S$.

\subsection{Semiclassical analysis}\label{subsec:semiclassical}
In what follows, we shall briefly recall the theory of semiclassical symbol classes and pseudodifferential operators on $\R^n$ and on general smooth manifolds. For a detailed introduction, we refer the reader to \cite[Chapters 9
and 14]{zworski} and \cite[Chapters 7 and 8]{dimassi-sjoestrand}.  Semiclassical analysis developed out of the theory of pseudodifferential operators, a thorough exposition of which can be found in \cite{shubin}. 
An important feature that distinguishes semiclassical analysis from usual pseudodifferential operator theory is that instead of the usual symbol functions and corresponding operators, one considers families of symbol functions and pseudodifferential operators indexed by a global parameter \[h\in (0,1].\]
Essentially, the definitions of those families are obtained from the usual definitions by substituting in the symbol functions the co-tangent space variable $\xi$ by $h\xi$. To begin, recall that a Lebesgue-measurable function $\mathfrak{m}:\R^{n}\to (0,\infty)$ is called \emph{order function} if there are constants $C,N>0$ such that
\[
\mathfrak{m}(v_1)\leq C\eklm{v_1-v_2}^N\mathfrak{m}(v_2)\quad \forall\; v_1,v_2\in \R^{n}.
\]
Here we used the notation $\eklm{v}:=\sqrt{1+\Vert v\Vert^2}$. If $\mathfrak{m}_1,\mathfrak{m}_2$ are order functions, then $\mathfrak{m}_1\mathfrak{m}_2$ is also an order function. For example, $\mathfrak{m}=1_{\R^n}$  and $\mathfrak{m}(v)=\eklm{v}^k$, $k>0$, are order functions. Let $\mathfrak{m}:\R^{n}\to (0,\infty)$ be an order function. For $\delta\in [0,\frac{1}{2})$ and $k\in \R$, we define the semiclassical symbol class $S^k_{\delta}(\mathfrak{m})$ as the set of all functions $s:\R^{n}\times(0,1]\to \C$ such that $ s(\cdot,h)\in \Cinft(\R^{n})$ for each $h\in (0,1]$ and for each non-negative $n$-dimensional multiindex $\alpha$, there is a constant $C_{\alpha,\delta,k}>0$ with
\bq
|\partial_v^\alpha s(v,h)|\leq C_{\alpha,\delta,k} \,\mathfrak{m}(v)\,h^{-\delta |\alpha|-k}\quad \forall\; (v,h)\in \R^{n}\times (0,1].\label{eq:symbclass7834}
\eq
We write
\[
S^k_h(\mathfrak{m}):=S_{0}^k(\mathfrak{m}),\qquad  S_{\delta}(\mathfrak{m}):=S^0_{\delta}(\mathfrak{m}),\qquad  S(\mathfrak{m}):=S_{0}^0(\mathfrak{m}).
\]
We call an element of a semiclassical symbol class a \emph{symbol function}. Furthermore, let us define
\[
S^{-\infty}(\mathfrak{m}):=\bigcap_{k\in \R}  S^k_{\delta}(\mathfrak{m}),\quad\delta\in [0,1/2)\text{ arbitrary}.
\]
This set is in fact well-defined (the intersection on the right hand side is independent of $\delta$). $S^{-\infty}(\mathfrak{m})$ is the set of functions $s:\R^{n}\times(0,1]\to \C$ such that $ s(\cdot,h)\in \Cinft(\R^{n})$ for each $h\in (0,1]$ and for each non-negative $n$-dimensional multiindex $\alpha$ and each $N\in \N$, there is a constant $C_{\alpha,N}>0$ with
\[
|\partial_v^\alpha s(v,h)|\leq C_{\alpha,N} \,\mathfrak{m}(v)\,h^N\quad \forall\; (v,h)\in \R^{n}\times (0,1].
\]
In order to recall the definition of semiclassical asymptotic series, let $\mathfrak{m}:\R^{n}\to (0,\infty)$ be an order function. Given $\delta\in [0,\frac{1}{2})$, a sequence $\{k_j\}_{j\in \N}\subset \R$ with $k_j\to -\infty$ as $j\to \infty$, a sequence $\{s_j\}_{j\in \N}$ with $s_j\in S^{k_j}_{\delta}(\mathfrak{m})$, and a symbol function $s\in S_{\delta}(\mathfrak{m})$, we say that $\{s_j\}_{j\in \N}$ is \emph{asymptotic to} $s$ in $S_{\delta}(\mathfrak{m})$, in short
\[
s\sim \sum_{j=0}^\infty s_j\quad\text{in }\,S_{\delta}(\mathfrak{m}),
\]
provided that for each $N\in \N$ one has 
$
s- \sum_{j=0}^N s_j\in S^{k_{N+1}}_{\delta}(\mathfrak{m}).
$
We denote by $\mathcal{S}(\R^n)$ the vector space of Schwartz functions on $\R^n$, equipped with the semi-norms
\[
|f|_{\alpha,\beta}:=\sup_{x\in \R^n}|x^\alpha \partial^\beta f(x)|
\]
and denote by $\mathcal{S}'(\R^n)$ the topological dual space of $\mathcal{S}(\R^n)$, i.e.\ the space of continuous linear functionals on $\mathcal{S}(\R^n)$, equipped with the weak-$\ast$ topology. Let $\mathfrak{m}:\R^{2n}\to (0,\infty)$ be an order function. For $s\in S^k_{\delta}(\mathfrak{m})$,  $f\in \mathcal{S}(\R^n)$, and $x\in \R^n$, define
\bq
\mathrm{Op}_h(s)(f)(x):=\frac 1 {(2\pi h)^n} \intop_{\R^n} \intop_{\R^n} e^{\frac i h (x-y)\cdot\eta}
s\Big(\frac{x+y}{2},\eta,h\Big) f(y) \d y \d \eta.\label{eq:defweyl}
\eq
Then, by \cite[Theorem 4.16]{zworski} and \cite[Theorem 7.8]{dimassi-sjoestrand}, the function $$\mathrm{Op}_h(s)(f): x\mapsto \mathrm{Op}_h(s)(f)(x)$$ is an element of $\mathcal{S}(\R^n)$ and the map
\[
\mathrm{Op}_h(s): \mathcal{S}(\R^n)\to \mathcal{S}(\R^n),\qquad f\mapsto \mathrm{Op}_h(s)(f),
\]
is a continuous linear operator. Moreover, by duality $\mathrm{Op}_h(s)$ extends to a continuous linear operator
\[
\mathrm{Op}_h(s): \mathcal{S}'(\R^n)\to \mathcal{S}'(\R^n).
\]
This so-called \emph{Weyl-quantization} is motivated by the fact that the
classical Hamiltonian $H(x,\xi)=\xi^2$ should correspond to the quantum Laplacian
$-h^2 \Delta$, and that real-valued symbol functions should correspond to symmetric or, more desirably, essentially self-adjoint operators. An operator $\mathcal{S}'(\R^n)\to \mathcal{S}'(\R^n)$ of the form (\ref{eq:defweyl}) is called a \emph{semiclassical pseudodifferential operator} on $\R^n$. We denote by $\mathrm{Op}_h(S^k_{\delta}(\mathfrak{m}))$ the set of semiclassical pseudodifferential operators that are quantizations of symbol functions in $S^k_{\delta}(\mathfrak{m})$. For the following important formula, we introduce the standard symplectic form $\sigma:\R^{2n}\times \R^{2n}\to \R$ given by $\sigma(x,\xi;y,\eta):=\xi\cdot y-x\cdot \eta$.
\begin{thm}[{Composition formula, \cite[Theorem 7.9]{dimassi-sjoestrand}}]\label{thm:compositionformula}
Let $\mathfrak{m}_1, \mathfrak{m}_2:\R^{2n}\to (0,\infty)$ be order functions and $s_j\in S_{\delta}(\mathfrak{m}_j)$. Then, there is a symbol function $s\in S_{\delta}(\mathfrak{m}_1\mathfrak{m}_2)$ such that $\mathrm{Op}_h(s_1)\circ \mathrm{Op}_h(s_2)=\mathrm{Op}_h(s)$, and
\bq
s\sim \sum_{k=0}^\infty \frac{1}{k!}\Big(\frac{ih}{2}\sigma(D_x,D_\xi;D_y,D_\eta)\Big)^k s_1(x,\xi,h)s_2(y,\eta,h)|_{y=x,\eta=\xi}\quad \text{ in }S_{\delta}(\mathfrak{m}_1\mathfrak{m}_2).\label{eq:composexpans}
\eq
\end{thm}
Let $s\in S^k_{\delta}(1_{\R^{2n}})$. Then, by \cite[Theorem 7.11]{dimassi-sjoestrand}, the operator $\mathrm{Op}_h(s): \mathcal{S}'(\R^n)\to \mathcal{S}'(\R^n)$ bi-restricts\footnote{Here, we are regarding $\L^2(\R^n)$ as a subset of $\mathcal{S}'(\R^n)$.} to a bounded linear operator $\mathrm{Op}_h(s)\in \B(\L^2(\R^n))$ which is essentially given by (\ref{eq:defweyl}), and there is a constant $C>0$ which is independent of $h$, such that
\bq
\norm{\mathrm{Op}_h(s)}_{\B(\L^2(\R^n))}\leq Ch^{-k}\quad\forall\; h\in (0,1].\label{eq:l2cont45}
\eq
It will be important for us to know when a semiclassical pseudodifferential operator is of trace class, and to estimate its trace norm. For these tasks, the following results from \cite[p.\ 113, Lemma 9.3, Theorem 9.4]{dimassi-sjoestrand} are very useful. Let $s\in S^k_{\delta}(1_{\R^{2n}})$ for some $\delta\in [0,\frac{1}{2})$, $k\in \R$, and suppose that the function $s(\cdot,h):\R^{2n}\to \C$, $(y,\eta)\mapsto s(y,\eta,h)$ fulfills 
\[
\sum_{|\alpha|\leq 2n+1}\norm{\partial^\alpha s(\cdot,h)}_{\L^1(\R^{2n})}<\infty.
\]
Then, the operator $\mathrm{Op}_h(s):\L^2(\R^n)\to \L^2(\R^n)$ is of trace class with trace norm
\[
\norm{\mathrm{Op}_h(s)}_{\mathrm{tr},\L^2(\R^n)}\leq C h^{-n}\sum_{|\alpha|\leq 2n+1}\norm{\partial^\alpha s(\cdot,h)}_{\L^1(\R^{2n})},
\]
where $C>0$ is independent of $h$, and its trace is given by
\bq
\mathrm{tr}\,_{\L^2(\R^n)}\mathrm{Op}_h(s)= \frac{1}{(2\pi h)^n}\intop_{\R^{2n}}s(y,\eta,h)\d y \d \eta.\label{eq:trace11111111}
\eq
Moreover, the integral kernel estimates on \cite[p.\ 113]{dimassi-sjoestrand} and the estimate proved there for the relation between standard quantization and Weyl quantization imply the following results. Choose $\phi\in \CT(\R^n)$ with support inside some compact set $K\subset \R^n$, and denote the operator $\L^2(\R^n)\to \L^2(\R^n)$ given by pointwise multiplication with $\phi$ by $\Phi$.
\begin{itemize}
\item Suppose that the function $s(y,\cdot,h): \eta\mapsto s(y,\eta,h)$ is a Schwartz function for each $y\in \R^n$ and each $h\in (0,1]$. Then, the operator $\Phi\circ \mathrm{Op}_h(s)$ is of trace class and its trace norm fulfills uniformly for $h\in (0,1]$ the estimate
\bq
\norm{\Phi\circ \mathrm{Op}_h(s)}_{\mathrm{tr},\L^2(\R^n)}\leq Ch^{-n}\sum_{|\alpha|\leq 2n+1}\norm{\partial^\alpha (\phi s)(\cdot,h)}_{\L^1(\R^{2n})},\label{eq:symboltraceestims2}
\eq
with a constant $C>0$ that is independent of $h$.
\item As a special case of the previous one, we have in particular: Suppose that the function $s(y,\cdot,h): \eta\mapsto s(y,\eta,h)$ is compactly supported in $\R^n$ for each $y\in \R^n$ and each $h\in (0,1]$, and the volume of the support of the function $s(y,\cdot,h)$ is bounded uniformly in $y\in \R^n$ by some $h$-dependent constant $C_h>0$. Then, 
\bq
\norm{\Phi\circ \mathrm{Op}_h(s)}_{\mathrm{tr},\L^2(\R^n)}\leq C_{\phi}C_hh^{-n}\sum_{|\alpha|\leq 2n+1}\max_{(y,\eta)\in K\times \R^n}|\partial^\alpha s(y,\eta,h)|,\label{eq:symboltraceestims}
\eq
with a constant $C_\phi>0$ that depends on $\phi$ but not on $h$.
\end{itemize}

Very useful in combination with the previous lines is also the following observation, which follows from the statements above and the composition formula (\ref{eq:composexpans}), compare \cite[Proposition 9.5]{dimassi-sjoestrand}. For $i\in \{1,2\}$, let $s_i\in S^{k_i}_{\delta}(1_{\R^{2n}})$ for some $\delta\in [0,\frac{1}{2})$, $k_i\in \R$, and suppose that for each $h\in (0,1]$ the function $s_1(\cdot,h):\R^{2n}\to \C$ is compactly supported inside the interior of some $h$-independent compactum $K\subset \R^{2n}$. Let $s_1\sharp s_2\in S_{\delta}^{k_1+k_2}$ be the symbol obtained from $s_1$ and $s_2$ by the composition formula (\ref{eq:composexpans}). Then for each $N\in \N$, $R>0$, and each non-negative $2n$-dimensional multiindex $\alpha$, there is a constant $C_{\alpha,N}>0$ such that for all $(y,\eta)\in \R^{2n}$ with $\mathrm{dist}((y,\eta),K)\geq R$, one has
\[
|\partial^\alpha (s_1\sharp s_2)(y,\eta,h)|\leq C_{\alpha,N}h^{N(1-\delta)-k_1-k_2-\delta |\alpha|} \mathrm{dist}((y,\eta),K)^{-N}\quad \;\forall\; h\in(0,1].
\]
As the function $(y,\eta)\mapsto \eklm{\mathrm{dist}((y,\eta),K)}^{-N}$ is in $\L^1(\R^{2n})$ if $N>2n$, we can combine the preceding results to get
\begin{cor}\label{cor:supercor}
For $i\in \{1,2\}$, let $s_i\in S^{k_i}_{\delta}(1_{\R^{2n}})$ for some $\delta\in [0,\frac{1}{2})$, $k_i\in \R$, and suppose that for each $h\in (0,1]$ the function $s_1(\cdot,h):\R^{2n}\to \C$ is compactly supported inside some $h$-independent compactum $K\subset \R^{2n}$. Then the operator $\mathrm{Op}_h(s_1\sharp s_2)$ is of trace class and
\bq
\norm{\mathrm{Op}_h(s_1\sharp s_2)}_{\mathrm{tr},\L^2(\R^n)}=\mathrm{O}\big(h^{-n-k_1-k_2-(2n+1)\delta}\big)\qquad \text{as }h\to 0.\label{eq:supercorollary}
\eq
\end{cor}
This corollary is important as it tells us that the trace norm of the composition of two semiclassical pseudodifferential operators, one of which has a symbol supported inside a fixed compactum, essentially depends only on the norm of the derivatives of the original two symbols near the compactum. Surely, Corollary \ref{cor:supercor} could be generalized to $h$-dependent compactums $K(h)$, but as we are mainly interested in a functional calculus for $h$-dependent functions whose support shrinks as $h\to 0$, it is no big loss of generality to assume that the shrinking happens inside a fixed $h$-independent compactum.

Let $\mathfrak{m}:\R^{2n}\to (0,\infty)$ be an order function and let $s\in S(\mathfrak{m})$. We call the symbol function $s$ $\mathfrak{m}$\emph{-elliptic} if there is a constant $\varepsilon>0$ such that $|s|\geq\varepsilon\,\mathfrak{m}$. Crucial for all what follows is the following result:
\begin{thm}[{Essential self-adjointness \cite[Prop.\ 8.5]{dimassi-sjoestrand}}]\label{thm:essself}
Let  $\mathfrak{m}:\R^{2n}\to (0,\infty)$ be an order function with $\mathfrak{m}\geq 1$, and let $s\in S(\mathfrak{m})$ be a real-valued symbol function such that $s+i$ is $\mathfrak{m}$-elliptic, where $i$ denotes the imaginary unit $\sqrt{-1}$. Then, there is a number $h_0\in (0,1]$ such that the operator $(\mathrm{Op}_h(s)+i)^{-1}\in \B(\L^2(\R^n))$ exists for each $h\in (0,h_0]$. Furthermore, the operator $\mathrm{Op}_h(s): \mathcal{S}(\R^n)\to \mathcal{S}(\R^n)\subset \L^2(\R^n)$ is essentially self-adjoint in $\L^2(\R^n)$ for each $h\in (0,h_0]$, and one obtains the unique self-adjoint extension $\overline{\mathrm{Op}_h(s)}$ by equipping $\mathrm{Op}_h(s)$ with the domain
\[
(\mathrm{Op}_h(s)+i)^{-1}\L^2(\R^n)\subset \L^2(\R^n).
\]
\end{thm}
For example, if $\mathfrak{m}(y,\eta)=\eklm{\eta}^2$ and $s(y,\eta,h)=\norm{\eta}^2$, then
\bq
(\mathrm{Op}_h(s)+i)^{-1}\L^2(\R^n)=\Sob_h^2(\R^n),\label{eq:samesobolevspaces}
\eq
where $\Sob_h^2(\R^n)$ is the semiclassical equivalent of the Sobolev space $\Sob^2(\R^n)$, see \cite[Thm.\ 8.10]{zworski}. 
Now, the known functional calculus in $\R^n$ for fixed $h$-independent functions is summarized in
\begin{thm}[{\cite[Theorem 8.7 and p.\ 103]{dimassi-sjoestrand}}]\label{thm:funccalcrn}
Let $s$ be a symbol function as in Theorem \ref{thm:essself} and let $f\in \CT(\R)$. Consider the operator $f\big(\overline{\mathrm{Op}_h(s)}\big):\L^2(\R^n)\to \L^2(\R^n)$ defined by the spectral calculus for unbounded self-adjoint operators. Then, there is a symbol function $a_{s,f}\in \bigcap_{k\in \N}S(\mathfrak{m}^{-k})$ and a number $h_0\in (0,1]$ such that for $h\in (0,h_0]$
\[
f\Big(\overline{\mathrm{Op}_h(s)}\Big)=\mathrm{Op}_h(a_{s,f}).
\]
Moreover, if $s$ fulfills $
s\sim \sum_{j=0}^\infty h^j s_j \,\text{ in } S(\mathfrak{m})
$ 
for some sequence $\{s_j\}_{j=0,1,2,\ldots}\subset S(\mathfrak{m})$, then there is a sequence of polynomials $\{q_{s,j}(y,\eta,t,h)\}_{j=0,1,2,\ldots}$ in one variable $t\in \R$ with coefficients $h$-dependent functions in $\Cinft(\R^{2n})$ and with $q_{s,0}\equiv 1$, such that
\[
a_{s,f}\sim \sum_{j=0}^\infty h^j a_{s,f,j} \,\text{ in } S(1/\mathfrak{m}), \quad a_{s,f,j}(y,\eta,h)=\frac{1}{(2j)!}\Big(\frac{\partial}{\partial t}\Big)^{2j}(q_{s,j}(y,\eta,t,h)f(t))_{t=s_0(y,\eta,h)}.
\]
In particular, $a_{s,f,0}=f\circ s_0$.
\end{thm}
\noindent A corollary is the following trace formula for semiclassical pseudodifferential operators in $\R^n$:
\begin{thm}[{\cite[Theorem 9.6]{dimassi-sjoestrand}}]\label{thm:traceformrn}
Let  $\mathfrak{m}:\R^{2n}\to (0,\infty)$ be an order function with $\mathfrak{m}\geq 1$, and let $s\in S(\mathfrak{m})$ be a real-valued symbol function such that $s+i$ is $\mathfrak{m}$-elliptic, with an asymptotic expansion $ 
s\sim \sum_{j=0}^\infty h^j s_j \,\text{ in } S(\mathfrak{m}),
$ 
where $\{s_j\}_{j=0,1,2,\ldots}\subset S(\mathfrak{m})$. Let $I\subset \R$ be a bounded open interval with
\[
\liminf_{\norm{v} \to +\infty}\mathrm{dist}(s(v,h),I)\geq C\qquad\forall\;h\in(0,1]
\]
for a constant $C>0$ which is independent of $h$, and let $f\in \CT(I)\subset \CT(\R)$ be given. Then, the operator $f\big(\overline{\mathrm{Op}_h(s)}\big):\L^2(\R^n)\to \L^2(\R^n)$ is of trace class for small $h$, and as $h\to 0$, its trace is asymptotically given by
\[
\mathrm{tr}_{\L^2(\R^n)}f\Big(\overline{\mathrm{Op}_h(s)}\Big)=\frac{1}{(2\pi h)^n}\intop_{\R^{2n}}f\big(s_0(y,\eta,h)\big)\d y \d \eta + \mathrm{O}\big(h^{-n+1}\big).
\]
\end{thm}
In order to introduce semiclassical pseudodifferential operators on general smooth manifolds, we need the following special type of symbol classes which is invariant under pullbacks along diffeomorphisms. For $m\in \R$ and $\delta\in[0,\frac{1}{2})$,  one sets
\begin{multline}
S_{\delta}^{m}(\R^{n}):=\Big\{a:\R^{2n}\times(0,1]\to \C: \; a(\cdot,h)\in \Cinft(\R^{2n})\;\forall\; h\in (0,1],\text{ and }  \forall\text{ multiindices } s,t\;\\ \exists\; C_{s,t}>0:  |\partial^s_x\partial^t_\xi a(x,\xi,h)|\leq
C_{s,t}\left<\xi\right>^{m-|t|}h^{-\delta(|s|+|t|)}\;\forall\,x\in \R^n,\;h\in (0,1]\Big\}.\label{eq:kohnnirenberg}
\end{multline}
Note that $S_{\delta}^{m}(\R^{n})\subset S_{\delta}(\mathfrak{m}_m)$, where $\mathfrak{m}_m:\R^{2n}\to (0,\infty)$ is given by $\mathfrak{m}_m(x,\xi):=\left<\xi\right>^m$, but the reverse inclusion is not true. The symbol classes (\ref{eq:kohnnirenberg}) generalize the classical Kohn-Nirenberg classes. In the literature one usually encounters only the case $\delta=0$. In our context it is natural to allow $\delta>0$, since the $h$-dependent functional calculus is primarily useful for functions whose derivatives have growing supremum norms as $h\to 0$. See \cite{dyatlov} for more applications of the symbol class (\ref{eq:kohnnirenberg}). Let now $M$ be a smooth manifold of dimension $n$, and let $\{(U_\alpha,\gamma_\alpha)\}_{\alpha\in\mathcal{A}}$, $\gamma_\alpha:M \supset U_\alpha \to V_\alpha\subset \R^{n}$, be an atlas for $M$. Then one defines
\begin{align}
\nonumber S_{\delta}^{m}(M):=&\Big\{a:T^*M\times (0,1]\to \C,\; a(\cdot, h)\in \Cinft(T^*M)\;\forall\;h\in(0,1],\\
&\quad (\gamma_\alpha^{-1}) ^*(\varphi_\alpha a)\in S_{\delta}^m(\R^{n})\;\,\forall\;\alpha\in\mathcal{A},\; \forall\;\varphi_\alpha\in \CT(U_\alpha)\Big\},\label{def:symbclass2}
\end{align}
where $(\gamma_\alpha^{-1})^\ast$ denotes the pullback\footnote{The pullback is defined as follows: 
First, one identifies $T^*V_\alpha$ with $V_\alpha\times\R^n$. Then, given $a:T^*M\times (0,1]\to \C$, the function $(\varphi_\alpha a)\circ \big(\gamma^{-1}_\alpha\times(\partial{\gamma_\alpha^{-1}})^T\times \1_{(0,1]}\big):V_\alpha\times\R^n\times(0,1]\to \C$ has compact support inside $V_\alpha$ in the first variable, and hence extends by zero to a function  
  $\R^{2n}\times(0,1]\to \C$ which is smooth for each fixed $h$. This function is defined to be $(\gamma_\alpha^{-1}) ^*(\varphi_\alpha a)$.} along $\gamma^{-1}_\alpha$. 
The definition is independent of the choice of atlas, and we call an element of 
$S_{\delta}^{m}(M)$ a \textit{symbol function}, similarly to the notion of symbol functions on $\R^n$ defined above. We use the short hand notations $$S_{\delta}^{-\infty}(M):=\bigcap_{m\in \R}S_{\delta}^{m}(M),\qquad S^m(M):=S_{0}^m(M),\qquad m\in \R\cup\{-\infty\}.$$For $m\in \R\cup\{-\infty\}$ and $\delta\in [0,\frac{1}{2})$, we call a $\C$-linear map $P: \CT(M)\to \Cinft(M)$ \emph{semiclassical pseudodifferential operator on $M$ of order $(m,\delta)$} if the following holds:
\begin{enumerate}
\item For some (and hence any) atlas $\{(U_\alpha,\gamma_\alpha)\}_{\alpha\in\mathcal{A}}$,
$\gamma_\alpha:M \supset U_\alpha \to V_\alpha\subset \R^{n}$ of $M$ there exists a collection of symbol functions $\{s_\alpha\}_{\alpha\in\mathcal{A}}\subset S_{\delta}^m(\R^{n})$ such that for any two functions $\varphi_{\alpha,1}, \varphi_{\alpha,2} \in \CT(U_\alpha)$, it holds
\[
\varphi_{\alpha,1} P(\varphi_{\alpha,2} f)=\varphi_{\alpha,1} \mathrm{Op}_h(s_\alpha)((\varphi_{\alpha,2} f)\circ\gamma_\alpha^{-1})\circ\gamma_\alpha.
\]
\item For all $\varphi_{1}, \varphi_{2} \in \CT(M)$ with $\supp \varphi_{1}\cap \supp \varphi_{2}=\emptyset$, one has
\[
\norm{\Phi_1\circ P \circ \Phi_2}_{\Sob^{-N}(M)\to \Sob^N(M)}=\mathrm{O}(h^\infty)\quad\forall \;N=0,1,2,\ldots,
\]
where $\Phi_j$ is given by pointwise multiplication with $\varphi_j$, and $\Sob^N(M)$ is the $N$-th Sobolev space.
\end{enumerate}
When $\delta=0$, we just say \emph{order} $m$ instead of \emph{order} $(m,0)$. We denote by $\Psi^m_{\delta} (M)$ the $\C$-linear space of all semiclassical pseudodifferential operators on $M$ of order $(m,\delta)$, and we write $$\Psi_h^{m}(M):=\Psi^m_{0}(M),\qquad \Psi_h^{-\infty}(M)= \bigcap_{m\in \Z} \Psi_h^{m}(M).$$ 
From the classical theorems about pseudodifferential operators one infers in particular the following relation between symbol functions and semiclassical pseudodifferential operators, see \cite[page 86]{hoermanderIII}, \cite[Theorem 14.1]{zworski}, \cite[page 383]{dyatlov}. There is a $\C$-linear map 
\begin{equation}
\Psi^m_{\delta} (M) \to S_{\delta}^{m}(M)/\big(h^{1-2\delta}S_{\delta}^{m-1}(M)\big),\quad P  \mapsto \sigma(P) \label{eq:sigma}
\end{equation}
which assigns to a semiclassical pseudodifferential operator its \emph{principal symbol}. Moreover, for each choice of atlas  $\{(U_\alpha,\gamma_\alpha)\}_{\alpha\in\mathcal{A}}$ of $M$ and a partition of unity $\{\varphi_\alpha\}_{\alpha \in \mathcal{A}}$ subordinate to $\{U_\alpha\}_{\alpha \in \mathcal{A}}$, there is a $\C$-linear map called \emph{quantization}, written
\begin{equation}
S_{\delta}^{m}(M) \to \Psi^m_{\delta} (M), \quad s \mapsto \mathrm{Op}_{h,\left\{U_\alpha,\varphi_\alpha\right\}_{\alpha \in \mathcal{A}}}(s).\label{eq:op}
\end{equation}
Any choice of such a map induces the same $\C$-linear bijection
\begin{align}
\Psi^m_{\delta} (M)/\big(h^{1-2\delta}\Psi^{m-1}_{\delta}(M)\big) & \begin{matrix}\sigma\\ \rightleftarrows \\ \mathrm{Op}_h \end{matrix} S_{\delta}^{m}(M)/\big(h^{1-2\delta}S_{\delta}^{m-1}(M)\big),\label{eq:qmaps}
\end{align}
which means in particular that the bijection exists and is independent from the choice of atlas and partition of unity. We will call an element in the quotient set $$S_{\delta}^{m}(M)/\big(h^{1-2\delta}S_{\delta}^{m-1}(M)\big)$$ a \textit{principal symbol}, whereas we call the elements of $S_{\delta}^{m}(M)$ \textit{symbol functions}, as introduced above. Operations on principal symbols such as pointwise multiplication with other principal symbols or smooth functions and composition with smooth functions are defined by performing the corresponding operations on the level of symbol functions. For a semiclassical pseudodifferential operator $A$ on $M$, we will use the notation 
\[
\sigma(A)=[a]
\]
to express that the principal symbol $\sigma(A)$ is the equivalence class in the quotient set $$S_{\delta}^{m}(M)/\big(h^{1-2\delta}S_{\delta}^{m-1}(M)\big)$$ defined by the symbol function $a\in S_{\delta}^{m}(M)$. Finally, returning to the setup introduced at the beginning, the known functional calculus for our Schr\"odinger operator $P(h)$ on the closed connected Riemannian manifold $M$ for a fixed $h$-independent function is summarized in the following
\begin{thm}[{\cite[Theorems 14.9 and 14.10]{zworski}}]
\label{thm:schwartzinfty} Let $f\in \mathcal{S}(\mathbb{R})$. Then, the operator $f(P(h))$, defined by the spectral theorem
for unbounded self-adjoint operators, is an element of $\Psi_h^{-\infty}(M)$. Furthermore, $f(P(h))$ extends to a bounded operator $f(P(h)):\L^2(M)\to \L^2(M)$ of trace class, and one has
\begin{equation}
\sigma\left(f(P(h))\right)=[f\circ p].\label{eq:2}
\end{equation}
As $h\to 0$, the trace of $f(P(h))$ is asymptotically given by
\[
\mathrm{tr}_{\L^2(M)}f(P(h))=\frac{1}{(2\pi h)^n}\intop_{T^*M}f\circ p \,\d(T^*M)+\mathrm{O}(h^{-n+1}).
\]
\end{thm}
\vspace*{-1em}
\qed 
\section{Results for $\R^n$}\label{sec:resultsrn}
In this section, we extend the results of Theorems \ref{thm:funccalcrn} and \ref{thm:traceformrn} to functions which depend on the semiclassical parameter $h$. The generalized theorems will then be used in the next section to prove a precise explicit trace formula for a semiclassical Schr\"odinger operator on a closed Riemannian manifold.

\subsection{Relating the functional and symbolic calculi}
We begin with the following
\begin{thm}\label{thm:semiclassfunccalcrn}
Let  $\mathfrak{m}:\R^{2n}\to (0,\infty)$ be an order function with $\mathfrak{m}\geq 1$, and let $s\in S(\mathfrak{m})$ be a real-valued symbol function such that $s+i$ is $\mathfrak{m}$-elliptic, where $i$ denotes the imaginary unit $\sqrt{-1}$. Choose $f_h\in \mathcal{S}^{\mathrm{comp}}_{\delta}$. Then, the operator $f_h\big(\overline{\mathrm{Op}_h(s)}\big):\L^2(\R^n)\to \L^2(\R^n)$, defined by the spectral calculus for unbounded self-adjoint operators, is a semiclassical pseudodifferential operator for small $h$. More precisely, there is a symbol function $a\in \bigcap_{k\in \N}S_{\delta}(\mathfrak{m}^{-k})$ and a number $h_0\in (0,1]$ such that for $h\in (0,h_0]$
\[
f_h\Big(\overline{\mathrm{Op}_h(s)}\Big)=\mathrm{Op}_h(a).
\]
Moreover, if $s$ fulfills $
s\sim \sum_{j=0}^\infty h^j s_j \,\text{ in } S(\mathfrak{m})
$ 
for some sequence $\{s_j\}_{j=0,1,2,\ldots}\subset S(\mathfrak{m})$, then there is an asymptotic expansion in $S_{\delta}(1/\mathfrak{m})$
\bq
a\sim \sum_{j=0}^\infty a_j,\qquad a_j\in S_{\delta}^{j(2\delta-1)}(1/\mathfrak{m}),\label{eq:expansion19129}
\eq
where
\bq
a_j(y,\eta,h)=\frac{1}{(2j)!}\Big(\frac{\partial}{\partial t}\Big)^{2j}\big(q_j(y,\eta,t,h)f_h(t)\big)_{t=s_0(y,\eta,h)}\label{eq:localsymbexp}
\eq
for a sequence of polynomials $\{q_j(y,\eta,t,h)\}_{j=0,1,2,\ldots}$ in one variable $t\in \R$ with coefficients being $h$-dependent functions in $\Cinft(\R^{2n})$ and satisfying $q_0\equiv 1$. In particular, $$a_0(y,\eta,h)=f_h(s_0(y,\eta,h)).$$
\end{thm}
\begin{proof}We will adapt the proof of Dimassi and Sj\"ostrand of Theorem \ref{thm:funccalcrn}, extending it to $h$-dependent functions $f_h\in\mathcal{S}^{\mathrm{comp}}_{\delta}$. Let us briefly recall the main steps in the proof of Theorem \ref{thm:funccalcrn}. First, one uses the \emph{Helffer-Sj\"ostrand formula} to express $f\big(\overline{\mathrm{Op}_h(s)}\big)$ as a complex integral which involves the resolvent of the operator $\overline{\mathrm{Op}_h(s)}$ and an almost analytic extension of $f$. Then one proves that the resolvent is a semiclassical pseudodifferential operator, and that its symbol has a certain asymptotic expansion whose terms are then plugged into the Helffer-Sj\"ostrand formula. That the resulting term-wise integrals exist and that they define elements of appropriate symbol classes is proved using the properties of the almost analytic extension of $f$. Finally, the precise algebraic form of the resulting symbol expansion for $f\big(\overline{\mathrm{Op}_h(s)}\big)$ is obtained by replacing the complex integrals in the expansion up to a negligible remainder by integrals that one can evaluate using the Cauchy integral formula. We will now precisely study the relevant steps in the proof of Theorem \ref{thm:funccalcrn} and check how they need to be generalized or modified to work also for $h$-dependent functions $f_h$ satisfying our regularity conditions. As mentioned before, the first step in the proof of Theorem \ref{thm:funccalcrn} is the \emph{Helffer-Sj\"ostrand formula}, see \cite[Theorem 8.1]{dimassi-sjoestrand}. Thus, let $P$ be a self-adjoint operator on a Hilbert space $\mathcal{H}$. Let $f\in {\mathrm C}^2_c(\R)$ and let $\tilde{f}\in {\mathrm C}^1_c(\C)$ be an extension of $f$ with $\bar \partial_z \tilde{f}(z)=\mathrm{O}(|\text{Im }z|)$, where $\bar \partial_z=\frac{1}{2}(\partial_x+i\partial_y)$ with the notation $z=x+iy$. Then
\bq
f(P)=\frac{-1}{\pi}\intop_\C \bar \partial_z \tilde{f}(z)(z-P)^{-1}\d z, \label{eq:helfsj}
\eq
where $dz$ denotes the Lebesgue measure on $\C$, and the integral is a Riemann integral for functions with values in $\B(\mathcal{H})$. The statement \eqref{eq:helfsj} is of course applicable to $f=f_h\in {\mathrm C}^2_c(\R)$ for each $h\in (0,1]$ separately. No generalization is needed here. The second key step in the proof of Theorem \ref{thm:funccalcrn} is the following Lemma, see \cite[Prop.\ 8.6]{dimassi-sjoestrand}. Let $s$ be a symbol function as in Theorem \ref{thm:essself}. Then for each $z\in \C$, $\text{Im }z\neq 0$,  there is a symbol function $r_z\in S(1_{\R^{2n}})$ such that 
\bq
\Big(z-\overline{\mathrm{Op}_h(s)}\Big)^{-1}=\mathrm{Op}_h(r_z). \label{eq:ophs}
\eq
The family $\{r_z\}$, indexed by $z$, has the property that for each $2n$-dimensional non-negative multiindex $\alpha$, there is a constant $C_\alpha>0$ such that for  $|z|\leq \text{const.}$, one has
\bq
|\partial_v^\alpha r_z(v,h)|\leq C_\alpha \max\Big(1, \frac{h^{1/2}}{|\text{Im }z|}\Big)^{2n+1}|\text{Im }z|^{-|\alpha|-1}\quad \forall\; v\in \R^{2n},\;h\in (0,h_0], \label{eq:restimates}
\eq
where the number $h_0\in (0,1]$ is the same as in Theorem \ref{thm:essself}. The lemma does not involve the function $f$ at all, so obviously it does not need to be modified when $f=f_h$. The third step, which in combination with the Helffer-Sj\"ostrand formula and (\ref{eq:ophs}), (\ref{eq:restimates}) yields that $f\big(\overline{\mathrm{Op}_h(s)}\big)$ is a semiclassical pseudodifferential operator for small $h$, is given by the assertion that
\bq
\intop_{|\text{Im }z|\leq h^\gamma} \bar \partial_z \tilde{f}(z)\, r_z\d z\in S^{-\infty}(1_{\R^{2n}})\quad \forall\;\gamma>0,\;\forall\;f\in \CT(\R). \label{eq:helfsjsymb}
\eq
It is proved using (\ref{eq:restimates}) and a particular choice for the extension map $f\mapsto \tilde f$  in the Helffer-Sj\"ostrand formula. Specifically, one considers the extension\footnote{The multiplication with a cutoff function $\chi(y)$ in front of the integral is only implicitly mentioned at \cite[p.\ 93/94]{dimassi-sjoestrand}.}
\bq
\tilde f(x+iy)=\frac{\psi_f(x)\chi(y)}{2\pi}\intop_\R e^{i(x+iy)\xi}\chi(y\xi)\mathcal{F}(f)(\xi)\d \xi,\label{eq:extension}
\eq
where $\mathcal{F}(f)$ is the Fourier transform of $f$, $\chi\in \CT(\R,[0,1])$ is equal to $1$ on $[-1,1]$, and  $\psi_f\in \CT(\R)$ is equal to $1$ in a neighborhood of $\supp f$. The main feature of the extension map $\CT(\R)\to \CT(\C)$ defined by (\ref{eq:extension}) is that the functions $\tilde f$ in its image are \emph{almost analytic}, meaning that for each $N\in\N$ there is a constant $C_N>0$ such that
\bq
|\bar \partial _z  \tilde f(z)|\leq C_N |\text{Im }z|^N\qquad \forall \;z\in \C. \label{eq:almostanalytic}
\eq
Now suppose that $f=f_h$. This time we cannot just apply the existing results separately to $f_h$ for each $h$, because (\ref{eq:helfsjsymb}) is a statement about uniform estimates in $h\in (0,1]$. Of course, for each individual $h$, we obtain an almost analytic extension $\tilde f_h$ for which (\ref{eq:almostanalytic}) holds. However, the constant $C_N$  appearing in the inequality for some $\tilde f_h$ depends on the function $f_h$, and so in particular $C_N=C_N(h)$ depends on $h$. Since we need estimates that are uniform for $h\in (0,1]$, we cannot directly use (\ref{eq:almostanalytic}) when $f=f_h$. Instead, we study the proof of (\ref{eq:almostanalytic}) to deduce a more precise estimate with constants that are independent of the function $f$. In order to do this, let us make the additional assumptions about the function $\psi_f$ that we have $|\psi_f|\leq 1, |\psi_f'|\leq 1$, and that $\psi_f\equiv 1$ in a closed interval $I_f=[m_f,M_f]\subset \R$ whose endpoints have distance $1$ to the support of $f$, and that $\psi_f=0$ outside $[m_f-2,M_f+2]$. 
 As in \cite[p.\ 94]{dimassi-sjoestrand}, one calculates for each $N\in \N$ and $f\in \CT(\R)$
\begin{multline}
\bar \partial _z  \tilde f(x+iy)=y^N\frac{i}{4\pi}\Big(\psi_f(x)\intop_\R e^{i(x+iy)\xi}\chi_N(y\xi)\xi^{N+1}\mathcal{F}(f)(\xi)\d \xi\\
+ \psi_f'(x)\intop_\R\intop_\R e^{i(x-\widetilde{x}+iy)\xi}\frac{\chi_N(y\xi)}{(\xi +i)^2}(i+D_{\widetilde x})^2D_{\widetilde x}^N\Big( \frac{f(\widetilde x)}{x-\widetilde{x}+iy} \Big)\d \widetilde{x} \d \xi
\Big)\quad \forall\; y\in [-1,1],\label{eq:estim564564}
\end{multline}
where we used the notation $\chi_N(t):= t^{-N}\chi'(t)$.\footnote{Note that $\chi'=0$ in a neighborhood of $0$.} Due to our assumptions on $\psi_f$, $|x-\widetilde x|$ is bounded from below by $1$ on the support of $\psi_f'(x)f(\widetilde x)$, and we obtain from  (\ref{eq:estim564564})
\bq
|\bar \partial _z  \tilde f(x+iy)| \leq |y|^NC_N\Big(\norm{\xi^{N+1}\mathcal{F}(f)}_{\L^1(\R)}
+\max_{0\leq j \leq N+2}\norm{f^{(j)}}_{\L^1(\R)}\Big)\quad \forall\; y\in [-1,1],\label{eq:estim987987987}
\eq
where $C_N>0$ is independent of $f$. Now we observe $(i\xi)^{N+1}\mathcal{F}(f)=\mathcal{F}\big(f^{(N+1)}\big)$, and in addition we note that for every Schwartz function $f$ on $\R$ and every $j\in \{0,1,2,\ldots\}$
\bq
\norm{\mathcal{F}\big(f^{(j)}\big)}_{\L^1(\R)}\leq C_j\max_{0\leq k \leq 2}\norm{f^{(j+k)}}_{\L^1(\R)},\label{eq:estim214698}
\eq
with $C_j>0$ independent of $f$, see e.g.\ \cite[Lemma 3.5]{zworski}. Moreover, for a function with support of finite volume, we have the standard integral estimate 
\bq
\norm{f^{(l)}}_{\L^1(\R)}\leq \mathrm{vol}(\supp f)\norm{f^{(l)}}_{\infty}\quad \forall\; l\in\{0,1,2,\ldots\}.\label{eq:estim34263856278}
\eq
Taking into account the estimates (\ref{eq:estim214698}) and (\ref{eq:estim34263856278}), the result (\ref{eq:estim987987987}) turns into
\bq
|\bar \partial _z  \tilde f(z)| \leq C_N |\text{Im } z|^N\mathrm{vol}(\supp f) \max_{0\leq j \leq N+3}\norm{f^{(j)}}_{\infty}\qquad \forall\;z\in \C,\;|\text{Im }z|\leq 1\label{eq:betterestim2424}
\eq
with new constants $C_N>0$ that are independent of $f$. The estimate (\ref{eq:betterestim2424}) is exactly the modified version of  (\ref{eq:almostanalytic}) that we were looking for. Setting $f=f_h$, we get for each $N\in \N$ and all $z\in \C$ with $|\text{Im }z|\leq 1$
\bq
|\bar \partial _z  \tilde f_h(z)| \leq C_N |\text{Im } z|^N\mathrm{vol}(\supp f_h) \max_{0\leq j \leq N+3}\norm{f_h^{(j)}}_{\infty}\qquad \forall\; h\in (0,1],\label{eq:betterestim4545}
\eq
with $C_N>0$ independent of $h$. We will now use (\ref{eq:betterestim4545}) to prove
\bq
\intop_{|\text{Im }z|\leq h^\gamma} \bar \partial_z \tilde{f}_h(z)\, r_z\d z\in S^{-\infty}(1_{\R^{2n}})\quad  \forall\;\gamma>\delta,\;f_h\in \mathcal{S}^{\mathrm{comp}}_{\delta}\label{eq:helfsjsymbnew}
\eq
which is slightly weaker than the $f_h$-version of (\ref{eq:helfsjsymb}), but sufficient for our purposes. Recall that the statement  (\ref{eq:helfsjsymbnew}) means that for each $h\in (0,1]$ the function on $\R^{2n}$ defined by the integral is smooth and one has for all $v\in \R^{2n},\; h\in (0,1],\;\gamma>\delta,\;N=0,1,\ldots:$
\bq
\bigg|\partial_v^\alpha\intop_{|\text{Im }z|\leq h^\gamma} \bar \partial_z \tilde{f}_h(z)\, r_z(v,h)\d z\bigg| \leq C_{N,\alpha,\gamma} h^N\label{eq:want2}
\eq
with constants $C_{N,\alpha,\gamma}>0$.
As $\partial_z \tilde{f}_h$ is compactly supported, and the integrand is smooth, we can interchange integration and differentiation, so that the function on $\R^{2n}$ defined by the integral is indeed smooth for each $h$. Let now $\gamma>\delta$. Then (\ref{eq:restimates}) and (\ref{eq:betterestim4545}) imply that there is a number $h_0\in (0,1]$ such that for each $N=0,1,2,\ldots$ and each $\alpha$ there is a $C_{N,\alpha}>0$ such that for all $h\in (0,h_0]$
\begin{align*}
&\bigg|\partial_v^\alpha\intop_{|\text{Im }z|\leq h^\gamma} \bar \partial_z \tilde{f}_h(z)\, r_z(v,h)\d z\bigg|
\leq\intop_{|\text{Im }z|\leq h^\gamma} \big|\bar \partial_z \tilde{f}_h(z)\, \partial_v^\alpha r_z(v,h)\big|\d z\\
&\leq C_{N,\alpha}\mathrm{vol}(\supp f_h) \max_{0\leq j \leq N+3}\norm{f_h^{(j)}}_{\infty} \hspace*{-0.5em}\intop_{\begin{matrix}\scriptstyle |\text{Im }z|\leq h^\gamma\\
\scriptstyle z\in \supp \bar \partial_z \tilde{f}_h\end{matrix}}\hspace*{-0.5em} \max\Big(1, \frac{h^{1/2}}{|\text{Im }z|}\Big)^{2n+1}|\text{Im }z|^{-|\alpha|-1}|\text{Im } z|^N\d z.
\end{align*}
For each $N\geq |\alpha|+2n+2$, we can estimate the final integral according to
\begin{align*}
&\intop_{\begin{matrix}\scriptstyle |\text{Im }z|\leq h^\gamma\\
\scriptstyle z\in \supp \bar \partial_z \tilde{f}_h\end{matrix}} \max\Big(1, \frac{h^{1/2}}{|\text{Im }z|}\Big)^{2n+1}|\text{Im }z|^{-|\alpha|-1}|\text{Im } z|^N\d z\\
&=\intop_{\begin{matrix}\scriptstyle h^{1/2} \leq |\text{Im }z|\leq h^\gamma\\
\scriptstyle z\in \supp \bar \partial_z \tilde{f}_h\end{matrix}} |\text{Im }z|^{-|\alpha|-1}|\text{Im } z|^N\d z \;+\; h^{n+1/2}\hspace{-0.5cm} \intop_{\begin{matrix}\scriptstyle |\text{Im }z|< h^{1/2}\\
\scriptstyle z\in \supp \bar \partial_z \tilde{f}_h\end{matrix}}|\text{Im }z|^{-|\alpha|-2n-2}|\text{Im } z|^N\d z\\
&\leq C'_{\alpha,N}\mathrm{vol}_\C\big( \supp \bar \partial_z \tilde{f}_h\big)\big(h^{\gamma(N-|\alpha|-1)}+h^{n+1/2(1+N-|\alpha|-2n-2)}  \big)\\
&\leq 2C'_{\alpha,N}\mathrm{vol}_\C\big(\supp \bar \partial_z \tilde{f}_h\big)\big(h^{\min(\gamma,1/2)N-\max(\gamma,1/2)(|\alpha|+1)} \big).
\end{align*}
Turning our attention to the term $ \mathrm{vol}_\C \text{ supp }\bar \partial_z \tilde{f}_h$, note that by construction of the almost analytic extension $\tilde{f}_h$ \bq
\supp \bar \partial_z \tilde{f}_h\subset  \supp  \tilde{f}_h \subset  (\supp \psi_{f_h})\times [L,L]i\subset \C,\label{eq:2893579}
\eq where $L>0$ is some constant depending only on the cutoff function $\chi$. Due to our assumptions on the function $\psi_{f_h}$ in the paragraph before (\ref{eq:estim564564}) one has
\bq
\mathrm{vol}\big(\supp \psi_{f_h}\big)\leq \mathrm{diam}(\supp f_h) + 2 + 2\label{eq:289357933}
\eq
and we obtain $\mathrm{vol}_\C  \big((\supp \psi_{f_h})\times [L,L]i\big)\leq 2 L (\mathrm{diam}(\supp f_h) + 4)$. Collecting all estimates together yields for $h\in (0,h_0]$ and $N\geq |\alpha|+2n+2$
\begin{multline}
\bigg|\partial_v^\alpha\intop_{|\text{Im }z|\leq h^\gamma} \bar \partial_z \tilde{f}_h(z)\, r_z(v,h)\d z\bigg|
\leq C_{N,\alpha} h^{\min(\gamma,1/2)N-\max(\gamma,1/2)(|\alpha|+1)}\\
\mathrm{vol}(\supp f_h)(1+\mathrm{diam}(\supp f_h)) \max_{0\leq j \leq N+3}\norm{f_h^{(j)}}_{\infty}\label{eq:intermediateresult1}
\end{multline}
for some new constant $C_{N,\alpha}$ which is independent of $h$. We now use the regularity conditions on the function $(t,h)\mapsto f_h(t)$ encoded in the assumption $f_h\in \mathcal{S}^{\mathrm{comp}}_{\delta}$. The condition that the function is in $S_{\delta}(1_\R)$ yields $$\max_{0\leq j \leq N+3}\norm{f_h^{(j)}}_{\infty}=\mathrm{O}\big(h^{-(N+3)\delta}\big)\qquad \text{as } h\to 0,$$
and because the diameter of the support of $f_h$ grows at most polynomially in $h^{-1}$ as $h\to 0$, there is a constant $r\geq 0$ such that $\mathrm{vol}(\supp f_h)(1+\mathrm{diam}(\supp f_h))=\mathrm{O}\big(h^{-r})$ as $h\to 0$. Thus, we conclude
\bq
\bigg|\partial_v^\alpha\intop_{|\text{Im }z|\leq h^\gamma} \bar \partial_z \tilde{f}_h(z)\, r_z(v,h)\d z\bigg|\leq C_{N,\alpha}\, h^{N(\min(\gamma,1/2)-\delta)-\max(\gamma,1/2)(|\alpha|+1)-3\delta -r}\label{eq:intermediateresult2}
\eq
with a new constant $C_{N,\alpha}$. Given $N'\in \N$, we can set $N(N'):=\lceil\frac{N'+\max(\gamma,1/2)(|\alpha|+1)+3\delta +r}{\min(\gamma,1/2)-\delta} \rceil$ to obtain
\bq
\bigg|\partial_v^\alpha\intop_{|\text{Im }z|\leq h^\gamma} \bar \partial_z \tilde{f}_h(z)\, r_z(v,h)\d z\bigg|\leq C_{N(N'),\alpha}\, h^{N'}\quad \forall \;h\in(0,h_0].\label{eq:intermediateresult3}
\eq
Thus, after performing a global rescaling $h':=h/h_0$, we have shown (\ref{eq:want2}), or equivalently (\ref{eq:helfsjsymbnew}). The next intermediate result in the proof of Theorem \ref{thm:funccalcrn} that we want to generalize involves the integral over the whole complex plane. Namely, one easily obtains
\bq
\intop_{\C} \bar \partial_z \tilde{f}(z)\, r_z\d z\in S(1_{\R^{2n}})\quad \forall\;f\in \CT(\R) \label{eq:easys0}
\eq
by taking into account that the integrand has compact support and estimating its $\L^\infty$-norm using (\ref{eq:restimates}) and (\ref{eq:almostanalytic}). Just as (\ref{eq:helfsjsymb}), (\ref{eq:easys0}) is a statement about uniform estimates in $h\in (0,1]$, so it does not directly generalize to $h$-dependent functions. We would like to prove
\bq
\intop_{\C} \bar \partial_z \tilde{f}_h(z)\, r_z\d z\in S_{\delta}(1_{\R^{2n}})\quad \forall\;f_h\in \mathcal{S}^{\mathrm{comp}}_{\delta}. \label{eq:desired0}
\eq
Let us try to prove (\ref{eq:desired0}) in the same way as (\ref{eq:easys0}) by estimating the $\L^\infty$-norm of the integrand and using that the integrand has compact support. With (\ref{eq:want2}), we can write
\begin{multline*}
\bigg|\partial_v^\alpha \intop_{\C} \bar \partial_z \tilde{f}_h(z)\, r_z(v,h)\d z\bigg|
=\bigg|\partial_v^\alpha \intop_{|\text{Im }z|\leq h^{1/2}} \bar \partial_z \tilde{f}_h(z)\, r_z(v,h)\d z+\partial_v^\alpha \intop_{|\text{Im }z|> h^{1/2}} \bar \partial_z \tilde{f}_h(z)\, r_z(v,h)\d z\bigg|\\
\leq \intop_{\begin{matrix}\scriptstyle |\text{Im }z|> h^{1/2}\\
\scriptstyle z\in \supp \bar \partial_z \tilde{f}_h\end{matrix}}\big| \bar \partial_z \tilde{f}_h(z)\,\partial_v^\alpha r_z(v,h)\big|\d z + \mathrm{O}(h^\infty),
\end{multline*}
the $\mathrm{O}(h^\infty)$ estimate being uniform in $v$. Using (\ref{eq:restimates}), (\ref{eq:betterestim4545}), (\ref{eq:2893579}), and (\ref{eq:289357933}), it follows that with $N=|\alpha| +1$
\begin{align*}
\bigg|\partial_v^\alpha \intop_{\C} \bar \partial_z \tilde{f}_h(z)\, r_z(v,h)&\d z\bigg|
\leq C_{N,\alpha}\mathrm{vol}(\supp f_h)\\
&\qquad\qquad\max_{0\leq j \leq N+3}\norm{f_h^{(j)}}_{\infty} \intop_{\begin{matrix}\scriptstyle |\text{Im }z|> h^{1/2}\\
\scriptstyle z\in \supp \bar \partial_z \tilde{f}_h\end{matrix}} |\text{Im }z|^{-|\alpha|-1}|\text{Im } z|^N\d z
 + \mathrm{O}(h^\infty)\\
&\leq  C_{|\alpha|+1,\alpha}\mathrm{vol}(\supp f_h) \max_{0\leq j \leq |\alpha|+1+3}\norm{f_h^{(j)}}_{\infty}  \text{vol supp }\bar \partial_z \tilde{f}_h + \mathrm{O}(h^\infty)\\
&\leq C_{\alpha,\delta}\mathrm{vol}(\supp f_h)(1+\mathrm{diam}(\supp f_h)) h^{-\delta(|\alpha|+4)} + \mathrm{O}(h^\infty)\\
&=\mathrm{O}\big(h^{-\delta(|\alpha|+4)-2r}\big),
\end{align*}
where $r>0$ is chosen such that the diameter of the support of $f_h$ is of order $h^{-r}$ as $h\to 0$. Thus, we arrive at the statement
\bq
\intop_{\C} \bar \partial_z \tilde{f}_h(z)\, r_z\d z\in S^{4\delta+2r}_{\delta}(1_{\R^{2n}})\quad \forall\;f_h\in \mathcal{S}^{\mathrm{comp}}_{\delta} \label{eq:notdesired0}
\eq
which is considerably weaker than (\ref{eq:desired0}). Temporarily, (\ref{eq:notdesired0}) will be sufficient to continue with the proof, and we will deduce (\ref{eq:desired0}) later. As in the proof of Theorem \ref{thm:funccalcrn}, we deduce from (\ref{eq:notdesired0})
\bq
\intop_{\C} \bar \partial_z \tilde{f}_h(z)\, r_z\d z\in S^{4\delta+2r}_{\delta}(\mathfrak{m}^{-k})\qquad \forall\;k\in\{0,1,2,\ldots\}\label{eq:newnotdesired}
\eq
by writing $f_{h,k}(t):=(t+i)^{k}f_h(t)$ and observing that $$f_h(\overline{\mathrm{Op}_h(s)})=\big(\overline{\mathrm{Op}_h(s)}+i\big)^{-k}f_{h,k}\big(\overline{\mathrm{Op}_h(s)}\big),$$
see \cite[Thm.\ 8.7]{dimassi-sjoestrand}. To proceed, fix some $\gamma\in (\delta,\frac{1}{2})$. It is shown in the proof of Theorem \ref{thm:funccalcrn} that if $|z|\geq h^\gamma$, $|z|\leq \text{const.}$, the function $r_z$ belongs to the symbol class $S_{\gamma}^\gamma(\mathfrak{m}^{-1})$ with estimates that are uniform in $z$, and $r_z$ has an expansion in $S_{\gamma}^\gamma(\mathfrak{m}^{-1})$ of the form
\[
r_z(y,\eta,h)\sim \sum_{j=0}^\infty h^j\sum_{k=0}^{2j} \frac{q_{j,k}(y,\eta,h)\,z^k}{(z-s_0(y,\eta,h))^{2j+1}},\qquad q_{j,k}\in S(\mathfrak{m}^{2j-k}),\qquad q_{0,0}\equiv 1,
\]
with uniform estimates on the domain $|z|\geq h^\gamma$, $|z|\leq \text{const.}$, see \cite[p.\ 102]{dimassi-sjoestrand}. Thus, by (\ref{eq:helfsj}), (\ref{eq:ophs}), (\ref{eq:want2}), and (\ref{eq:newnotdesired}) one obtains $f_h(\overline{\mathrm{Op}_h(s)})=\mathrm{Op}_h(a)$ with $a\in S^{4\delta+2r}_{\delta}(\mathfrak{m}^{-1})$ having an asymptotic expansion in $S^{4\delta+2r}_{\delta}(\mathfrak{m}^{-1})$ 
\bq
a\sim \sum_{j=0}^\infty h^j\, \widetilde{a}_j,\qquad \widetilde{a}_j(y,\eta,h)=\frac{-1}{\pi}\intop_{|z|\geq h^\gamma}\bar \partial_z \tilde{f}_h(z)\,\frac{q_j(y,\eta,z,h)}{(z-s_0(y,\eta,h))^{2j+1}}\d z,\label{eq:expansion27}
\eq
where we wrote $q_j(y,\eta,z,h):=\sum_{k=0}^{2j}q_{j,k}(y,\eta,h)\,z^k$. 
For the same reason why (\ref{eq:helfsjsymbnew}) holds, one can replace each $\widetilde{a}_j$ up to an error in $S^{-\infty}(\mathfrak{m}^{-1})$ by
\begin{align*}
a_j(y,\eta,h)&=\frac{-1}{\pi}\intop_{\C}\bar \partial_z \tilde{f}_h(z)\,\frac{q_j(y,\eta,z,h)}{(z-s_0(y,\eta,h))^{2j+1}}\d z\\
&=\frac{1}{(2j)!}\Big(\frac{\partial}{\partial t}\Big)^{2j}\big(q_j(y,\eta,t,h)f_h(t)\big)_{t=s_0(y,\eta,h)},
\end{align*}
where the evaluation of the complex integral between the first and the second line is as in \cite[(8.16) on p.\ 103]{dimassi-sjoestrand}. 
We obtain 
\bq
a\sim \sum_{j=0}^\infty h^j\, a_j\quad \text{in }S^{4\delta+2r}_{\delta}(\mathfrak{m}^{-1}).\label{eq:expansion28}
\eq
Now, since $f_h$ is an element of $S_{\delta}(1_\R)$ and only derivatives of $f_h$ of order at most $2j$ occur in $a_j$, we conclude that $a_j\in S^{2j\delta}_{\delta}(\mathfrak{m}^{-1})$. Therefore, the expansion (\ref{eq:expansion28}) implies that $a$ is in fact an element of $S_{\delta}(\mathfrak{m}^{-1})\subset S^{4\delta+2r}_{\delta}(\mathfrak{m}^{-1})$ and has the same expansion in $S_{\delta}(\mathfrak{m}^{-1})$. Thus, the evaluation of the complex integrals in the individual terms of the expansion (\ref{eq:expansion27}) has finally provided a proof for (\ref{eq:desired0}). Just as we deduced (\ref{eq:newnotdesired}) from (\ref{eq:notdesired0}), we deduce from (\ref{eq:desired0}) that $a\in \bigcap_{k\in \N}S_{\delta}(\mathfrak{m}^{-k})$.
\end{proof}
As a corollary, we get a semiclassical trace formula that generalizes Theorem \ref{thm:traceformrn}.
\begin{cor}
Let  $\mathfrak{m}:\R^{2n}\to (0,\infty)$ be an order function with $\mathfrak{m}\geq 1$, and $s\in S(\mathfrak{m})$ be a real-valued symbol function with an asymptotic expansion \[
s\sim \sum_{j=0}^\infty h^j s_j \,\text{ in } S(\mathfrak{m})
\]
such that $s+i$ is $\mathfrak{m}$-elliptic. Let $I\subset \R$ be a bounded open interval with
\[
\liminf_{\norm{v} \to +\infty}\mathrm{dist}(s(v,h),I)\geq C\qquad \forall\; h\in (0,1]\]
for a constant $C>0$ that is independent of $h$, and let $f_h\in \CT(I)\subset \CT(\R)$ be given such that the function $(t,h)\mapsto f_h(t)$ is an element of the symbol class $S_{\delta}(1_\R)$ for some $\delta\in [0,\frac{1}{2})$. 
Then, the operator $f_h\big(\overline{\mathrm{Op}_h(s)}\big):\L^2(\R^n)\to \L^2(\R^n)$ is of trace class for small $h$, and as $h\to 0$ one has
\bq
\mathrm{tr}_{\L^2(\R^n)}f_h\Big(\overline{\mathrm{Op}_h(s)}\Big)=\frac{1}{(2\pi h)^n}\intop_{\R^{2n}}f_h\big(s_0(y,\eta,h)\big)\d y \d \eta 
+ \mathrm{O}\Big(h^{1-2\delta-n}\mathrm{vol}_{\,\R^{2n}}\big(\supp f_h\circ s_0(\cdot,h)\big)\Big).
\eq
\end{cor}
\begin{proof}
In view of Theorem \ref{thm:semiclassfunccalcrn} and Corollary \ref{cor:supercor}, one can prove the statements of the theorem by complete analogy to \cite[proof of Theorem 9.6]{dimassi-sjoestrand}.
\end{proof}

\section{Results for closed Riemannian manifolds}\label{sec:closedriem}
In this section, we generalize Theorem \ref{thm:schwartzinfty} to functions which depend on $h$, and we establish explicit statements adapted to proving trace formulas. For the whole section, let us fix the following setup. As introduced before, let $M$ be a closed connected Riemannian manifold of dimension $n$. We choose a finite atlas $\{U_\alpha,\gamma_\alpha\}_{\alpha \in \mathcal{A}}$ with charts $\gamma_\alpha: U_\alpha\stackrel{\simeq}{\to} \R^n$, $U_\alpha\subset M$ open, and a subordinate partition of unity $\{\varphi_\alpha\}_{\alpha\in\mathcal{A}}$. For each $\alpha\in\mathcal{A}$, we choose in addition a compact set $K_\alpha\subset \R^n$ such that $\supp \varphi_\alpha\circ\gamma_\alpha^{-1}\subset\,  \mathrm{Int}(K_\alpha)$, and three cutoff functions $\overline{\varphi}_\alpha,\overline{\overline \varphi}_\alpha,\overline{\overline{\overline \varphi}}_\alpha\in \Cinft(M)$ with supports contained in $\gamma^{-1}_\alpha(\mathrm{Int}(K_\alpha))\subset U_\alpha$ and with $\overline{\varphi}_{\alpha}\equiv 1$ on $\supp \varphi_\alpha$, $\overline{\overline \varphi}_{\alpha}\equiv 1$ on $\supp \overline{\varphi}_{\alpha}$, and $\overline{\overline{\overline \varphi}}_{\alpha}\equiv 1$ on $\supp \overline{\overline \varphi}_{\alpha}$. For a point $x\in U_\alpha$, let $y=(y_1,\ldots,y_n)\in \R^n$ be the coordinates of the point $\gamma_\alpha(x)$. Furthermore, we have a local metric $g_\alpha:=(g_\alpha^{ij})$ with coefficients $g_\alpha^{ij}:\R^n\rightarrow \R$ and inverse matrix $(g_{ij}^\alpha)$, together with an associated volume density $\mathrm{Vol}_{g_\alpha}(y):=\sqrt{\text{det }g_\alpha(y)}$. Note that $\lim_{\Vert y\Vert \to\infty}\mathrm{Vol}_{g_\alpha}(y)=0$, since otherwise $U_\alpha$ would have infinite Riemannian volume in contradiction to the compactness of $M$.  It follows that the positive function $y\to \mathrm{Vol}_{g_\alpha}(y)$ is bounded. 
\subsection{Technical preparations}
The Schr\"odinger operator $\breve P(h)$ acts on a function $f\in \CT(U_\alpha)\subset \Cinft(M)$ by the formula
\begin{align}
\nonumber \breve P(h)(f)(x)=\breve{S}_{\alpha}(h) (f\circ \gamma_\alpha^{-1})(y) := \frac {-h^2} {\mathrm{Vol}_{g_\alpha}(y)}&\sum_{i,j=1}^n \frac \gd{\gd y_j} \left (g^\alpha_{ij} \mathrm{Vol}_{g_\alpha} \frac{\gd (f\circ \gamma_\alpha^{-1})}{\gd y_i} \right)(y)\\ &+(V\circ \gamma_\alpha ^{-1})(y)\cdot(f\circ \gamma_\alpha^{-1})(y)\label{eq:deflocalp333}
\end{align}
for $x\in U_\alpha$, and $\breve P(h)(f)(x)=0$ for $x\in M- U_\alpha$.
 The so defined operator $$\breve {S}_{\alpha}(h):\CT(\R^n)\to \CT(\R^n)$$ is a second order elliptic differential operator on $\R^n$, in the sense that its principal symbol is nowhere $0$. However, $\breve {S}_{\alpha}(h)$ is not uniformly elliptic in the sense that its principal symbol is bounded away from $0$, because the coefficients $g_\alpha^{ij}$ and $g^\alpha_{ij}$ can tend to zero towards infinity. To circumvent this problem, let $\tau_\alpha \in \CT(\R^n,[0,1])$ be a function which fulfills $\tau_\alpha\equiv 1$ in a neighborhood of $K_\alpha$, and define a new differential operator $\CT(\R^n)\to \CT(\R^n)$ by
\bq
\breve{P}_{\alpha}(h):= \tau_\alpha\breve{S}_{\alpha}(h) + (1-\tau_\alpha)(-h^2\Delta),\label{eq:newop}
\eq
where $\Delta=\sum_{i=1}^n\frac{\partial^2}{\partial y_i^2}$. Clearly, $\breve{P}_{\alpha}(h)$ agrees with $\breve {S}_{\alpha}(h)$ on functions supported inside $K_\alpha$. The reason why we introduced the new operator $\breve{P}_{\alpha}(h)$ is
\begin{lem}\label{lem:symbfunc}
Let $\mathfrak{m}:\R^2\to (0,\infty)$ be the order function given by $(y,\eta)\mapsto \eklm{\eta}^2$. Then, for each $\alpha$, one has $\breve{P}_{\alpha}(h)=\mathrm{Op}_h(p_\alpha)$ for a real-valued symbol function $p_\alpha\in S(\mathfrak{m})$ such that $p_\alpha+i$ is $\mathfrak{m}$-elliptic.  Furthermore, $\breve{P}_{\alpha}(h)$  has a unique self-adjoint extension $P_\alpha(h):\Sob_h^2(\R^n)\to \L^2(\R^n)$, and for $z\in \C$, $\mathrm{Im}\;z\neq 0$, the resolvent $(P_\alpha(h)-z)^{-1}:\L^2(\R^n)\to \Sob_h^2(\R^n)$ exists as a bounded operator. 
\end{lem}
\begin{proof}Fix $\alpha\in \mathcal{A}$ and note that as $M$ is compact we can assume without loss of generality that all the coefficients $g_\alpha^{ij}, g^\alpha_{ij}:\R^n\to \R$ and their derivatives are bounded. By (\ref{eq:deflocalp333}) and (\ref{eq:newop}),
\[
\breve{P}_{\alpha}(h)=\mathrm{Op}_h(p_\alpha)
\]
for a function $p_\alpha\in\Cinft(\R^{2n}\times(0,1])$ of the form
\bq
p_\alpha(y,\eta,h)=\underbrace{\sum_{i,j=1}^n p^{ij}_\alpha(y)\eta_i\eta_j+ V_\alpha(y)}_{=:p_{\alpha,0}(y,\eta)} + h\sum_{i=1}^np^i_\alpha(y)\eta_i,\qquad V_\alpha:=V\circ\gamma_
\alpha^{-1},\label{eq:symbfunclem}
\eq
where 
\[
p^{i,j}_\alpha(y)=\tau_\alpha(y) g^{ij}_\alpha(y)+(1-\tau_\alpha(y))\delta^{ij},\qquad p_\alpha^i\in \Cinft(\R^n,\R), 
\]
and the functions $p_\alpha^i$ and $V_\alpha$ are bounded. Here, $\delta^{ij}$ is the \emph{Kronecker delta}. Since all the coefficients in the polynomial $p_\alpha$ and all of their derivatives are bounded functions, there is for each non-negative $2n$-dimensional multiindex $\beta$ a constant $C_\beta>0$ such that $|\partial^\beta p_\alpha(y,\eta,h)|\leq C_\beta  \eklm{\eta}^2$ holds for all $(y,\eta)\in \R^{2n}$ and all $h\in (0,1]$. Thus, we conclude that $p_\alpha\in S(\mathfrak{m})$. It remains to show that there is some constant $\varepsilon_\alpha>0$ such that $|p_\alpha+i|\geq \varepsilon_\alpha\mathfrak{m}$. Let $y\in \R^n$. Since $g_\alpha(y)$ is a norm-induced metric on $\R^n$ and all norm-induced metrics on $\R^n$ are equivalent, it holds $\eta^T g_\alpha(y)\eta \geq c_\alpha(y)|\eta|^2$ for some $c_\alpha(y)>0$. Clearly, the function $\R^n\to (0,\infty)$, $y\mapsto c_\alpha(y)$, is smooth and thus assumes a minimum $m_\alpha$ on the compact support of $\tau_\alpha$. It follows 
\bq
\Big|\sum_{i,j}p^{i,j}_\alpha(y)\eta_i\eta_j\Big| \geq \min(1,m_\alpha) |\eta|^2\qquad \forall\;y\in \R^n.\label{eq:123413423413}
\eq
Now, recall that $p^i_\alpha$ and $V_\alpha$ are bounded functions, which in view of (\ref{eq:123413423413}) implies that we can find a constant $r_\alpha>0$ such that $\big|h\sum_{i}p^i_\alpha(y) \eta_i+V_\alpha(y) \big|< \frac{1}{2}\big|\sum_{i,j}p^{i,j}_\alpha(y)\eta_i\eta_j\big| $ holds for all $\eta$ with $|\eta|>r_\alpha$ and all $y\in \R^n$, $h\in(0,1]$. Thus, we conclude for $|\eta|>r_\alpha$:
$$\big|p_\alpha(y,\eta,h) +i \big|^2\geq  \Big(\frac{1}{2}\min(1,m_\alpha) |\eta|^2 \Big)^2 + 1\qquad \forall\;y\in \R^n,\;h\in(0,1].$$
Now, choose $R_\alpha\geq r_\alpha$ large enough and $C_\alpha>0$ small enough such that $$\Big(\frac{1}{2}\min(1,m_\alpha) |\eta|^2 \Big)^2 + 1\geq C_\alpha^2 \big(|\eta|^2+1\big)^2$$ for all $|\eta|\geq R_\alpha$. Then $\big|p_\alpha(y,\eta,h) + i\big|\geq  C_\alpha\eklm{\eta}^2$ for all $|\eta|\geq R_\alpha$ and all $y\in \R^n$, $h\in (0,1]$. To obtain an analogous statement also for $|\eta|\leq R_\alpha$, note that we trivially have $|p_\alpha +i |\geq 1$, because $p_\alpha$ is real-valued. Assuming w.l.o.g.\ that $R_\alpha\geq 1$, we get
$$\big|p_\alpha(y,\eta,h) + i\big|\geq 1=\frac{2R_\alpha^2}{2R_\alpha^2}\geq \frac{1}{2R_\alpha^2}|\eta|^2+ \frac{1}{2}\geq \frac{1}{2R_\alpha^2}|\eta|^2+ \frac{1}{2R_\alpha^2}=\frac{1}{2R_\alpha^2}\big(|\eta|^2+1).$$
We obtain for arbitrary $(y,\eta,h)$ that
$\big|p_\alpha(y,\eta,h) +i \big|\geq  \widetilde{C}_\alpha\eklm{\eta}^2$ with $\widetilde{C}_\alpha:=\min\big(C_\alpha,\frac{1}{2R_\alpha^2}\big)$, so that we are done with the proof that $p_\alpha + i$ is $\mathfrak{m}$-elliptic. The remaining statements of the lemma follow from Theorem \ref{thm:essself} and the observation (\ref{eq:samesobolevspaces}).
\end{proof}
\begin{lem}\label{lem:pullbackfunctn} For each $\alpha\in\mathcal{A}$, define the vector space
\[
\L_{\mathrm{comp},\alpha}^2(M):=\big\{f\in\L^2(M),\;\mathrm{ess.}\; \mathrm{supp}\; f\circ \gamma_\alpha^{-1}\subset K_\alpha\big\},
\]
and equip it with the norm induced from $\L^2(M)$. Then
\[
{\Gamma}_\alpha^*: \B(\L^2(\R^n)) \to \B(\L_{\mathrm{comp},\alpha}^2(M),\L^2(M)),\qquad A \mapsto \big(f\mapsto \overline{A(f\circ\gamma^{-1}_\alpha)\circ \gamma_\alpha}^0\big)
\]
is a bounded linear operator, where $\overline{u}^0$ denotes continuation of the function $u$ by zero outside $U_\alpha$, and the vector spaces $\B(\L^2(M))$ and $\B(\L_{\mathrm{comp},\alpha}^2(M),\L^2(M))$ are each equipped with the operator norm.
\end{lem}
\begin{proof}
First, note that a function $f\in \L_{\mathrm{comp},\alpha}^2(M)$ indeed pulls back to a function $f\circ \gamma_\alpha^{-1} $ in $\L^2(\R^n)$: As the volume density $\mathrm{Vol}_{g_\alpha}$ is bounded, the only critical issue here is decay at infinity, and $f\circ \gamma_\alpha^{-1}$ has compact support. It now suffices to show that there are constants $C_\alpha,C'_\alpha>0$ such that
\begin{align}
\norm{f\circ\gamma_\alpha^{-1}}_{\L^2(\R^n)}&\leq C_\alpha\norm{f}_{\L^2(M)} \qquad \forall\;f\in \L_{\mathrm{comp},\alpha}^2(M), \label{eq:firstts} \\
\norm{a\circ\gamma_\alpha}_{\L^2(M)}&\leq C'_\alpha\norm{a}_{\L^2(\R^n)}\qquad  \forall\;a\in \L^2(\R^n). \label{eq:sects}
\end{align}
Then it will follow that 
\bq
\norm{{\Gamma}_\alpha^*(A)}_{\B(\L_{\mathrm{comp},\alpha}^2(M),\L^2(M))}\leq C_\alpha C_\alpha'\norm{A}_{\B(\L^2(\R^n))}\qquad \forall\;A\in \B(\L^2(\R^n)).\label{eq:contgamma}
\eq
The first relation (\ref{eq:firstts}) can be proved easily by observing that for a function $f\in \L_{\mathrm{comp},\alpha}^2(M)$  one has
\bqn
\norm{f}^2_{\L^2(M)} \geq\intop_{U_\alpha} |f|^2\d M=\intop_{\R^n} | f(\gamma_\alpha^{-1}(y))|^2 \mathrm{Vol}_{g_\alpha}(y)\d y
=\intop_{K_\alpha} | f(\gamma_\alpha^{-1}(y))|^2 \mathrm{Vol}_{g_\alpha}(y)\d y.
\eqn
Setting \[C_\alpha^{-2}:=\min_{y\in K_\alpha}\mathrm{Vol}_{g_\alpha}(y)>0,\]
it follows
\bqn
\norm{f}^2_{\L^2(M)}\geq C_\alpha^{-2} \intop_{K_\alpha} | f(\gamma_\alpha^{-1}(y))|^2 \d y=C_\alpha^{-2} \intop_{\R^n} |  f(\gamma_\alpha^{-1}(y))|^2 \d y
\equiv C_\alpha^{-2} \norm{f\circ\gamma_\alpha^{-1}}^2_{\L^2(\R^n)}.
\eqn
The assertion (\ref{eq:sects}) follows from the boundedness of the function $\mathrm{Vol}_{g_\alpha}.$ Namely, for $a\in \L^2(\R^n)$ we get
\begin{multline*}
\norm{a\circ\gamma_\alpha}^2_{\L^2(M)}\equiv\intop_M |a\circ\gamma_\alpha|^2\d M =\intop_{U_\alpha} |a\circ\gamma_\alpha|^2\d M=\intop_{\R^n} |a(y)|^2 \mathrm{Vol}_{g_\alpha}(y)\d y \\
\leq \underbrace{\Big(\sup_{y\in \R^n}\mathrm{Vol}_{g_\alpha}(y)\Big)}_{=:C_\alpha'^2}\intop_{\;\R^n} |a(y)|^2 \d y\equiv C_\alpha'^2\norm{a}^2_{\L^2(\R^n)}.
\end{multline*}
\end{proof}
The following resolvent estimate will be very useful.
\begin{lem}\label{lem:resolvestim}
For $z\in \C$, $\mathrm{Im}\;z\neq 0$, consider the resolvent $(P_\alpha(h)-z)^{-1}$ from Lemma \ref{lem:symbfunc} for some $\alpha\in \mathcal{A}$. Let $r,s\in \CT(\R^n)$ have disjoint supports and associated multiplication operators $\Phi_r,\Phi_s:\L^2(\R^n)\to \L^2(\R^n)$. Then, for each $N\in \N$ there is a constant $C_N>0$, depending on $r$ and $s$, such that  for $|z|\leq$ const.\ one has the estimate
\[
\norm{\Phi_r\circ(P_\alpha(h)-z)^{-1}\circ\Phi_s}_{\B(\L^2(\R^n))}\leq C_Nh^N |\mathrm{Im}\;z|^{-N-1}.
\]
\end{lem}
\begin{proof}We owe the trick used in this proof to Maciej Zworski. Fix some $N\in \N$. Set $r_1:=r$ and choose functions $r_2\ldots,r_{N}$ which fulfill $r_i\equiv 1$ on $\supp r_{i-1}$ and $\supp r_{i}\cap \supp s=\emptyset$ for $i\in \{2,\ldots,N\}$. Let $\Phi_{i}:\L^2(\R^n)\to \L^2(\R^n)$ be the pointwise multiplication operator associated to $r_i$. Then we have $\Phi_{1}\circ\Phi_{2}\circ\cdots\circ\Phi_{N}=\Phi_{r}$. Next, observe that for any operator $A$ on $\L^2(\R^n)$ the commutators $[P_\alpha(h)-z,A]$ and $[P_\alpha(h),A]$ agree, since $z$ is just a multiple of the identity operator and hence has zero commutator. In addition, note that $\Phi_{i}\circ\Phi_{s}=0$ for all $i\in \{1,\ldots,N\}$, by choice of the functions $r_i$, and that $\Phi_k\circ(P_\alpha(h)-z)\circ({\bf 1}-\Phi_{k+1})=0$, since $P_\alpha(h)-z$ is a differential operator and as such a local operator. With those observations, one verifies easily
 \begin{multline*}
\hspace{-0.35cm}(P_\alpha(h)-z)^{-1}\circ[P_\alpha(h),\Phi_{1}]\circ(P_\alpha(h)-z)^{-1}\circ[P_\alpha(h),\Phi_{2}]\circ\\
\cdots\circ(P_\alpha(h)-z)^{-1}\circ[P_\alpha(h),\Phi_{N}]\circ(P_\alpha(h)-z)^{-1}\circ\Phi_s  \\
=\Phi_{1}\circ\Phi_{2}\circ\cdots\circ\Phi_{N}\circ(P_\alpha(h)-z)^{-1}\circ\Phi_s=\Phi_r\circ(P_\alpha(h)-z)^{-1}\circ\Phi_s.
\end{multline*}
Each commutator is independent of $z$, and by \cite[top of p.\ 102]{dimassi-sjoestrand} we have for $|z|\leq$ const.\ the estimate $$\norm{[P_\alpha(h),\Phi_{i}]\circ(P_\alpha(h)-z)^{-1}}_{\B(\L^2(\R^n))}= \mathrm{O}(h|\text{Im }z|^{-1})\qquad \forall\; i\in\{1,\ldots,N\}$$
and $$\norm{(P_\alpha(h)-z)^{-1}}_{\B(\L^2(\R^n))}= \mathrm{O}(|\text{Im }z|^{-1}).$$
Therefore, we can conclude that
\begin{align*}
\norm{\Phi_r\circ(P_\alpha(h)-z)^{-1}\circ\Phi_s}_{\B(\L^2(\R^n))}&\leq \norm{(P_\alpha(h)-z)^{-1}}_{\B(\L^2(\R^n))}\norm{[P_\alpha(h),\Phi_{1}]\circ(P_\alpha(h)-z)^{-1}}_{\B(\L^2(\R^n))}\\
&\qquad \cdots\norm{[P(h),\Phi_{N}]\circ(P_\alpha(h)-z)^{-1}}_{\B(\L^2(\R^n))}\norm{\Phi_s}_{\B(\L^2(\R^n))}\\
& \leq C_N h^N (|\text{Im }z|^{-1})^{N+1}.
\end{align*}
\end{proof}

\subsection{Operator norm estimates}

We can now state and prove our first theorem about the semiclassical functional calculus for $h$-dependent functions and the Schr\"odinger operator $P(h)$ on $M$ with associated Hamiltonian $p$, relating the functional and symbolic calculi with operator norm remainder estimates. The theorem we will prove is actually much more explicit than Result \ref{res:partIresult2} from the summary in the introduction, where it was stated in a condensed form. Choose $\rho_h\in \mathcal{S}^{\mathrm{comp}}_{\delta}$.  We then obtain for each $h\in (0,1]$ an operator $\rho_h(P(h))\in \B(\L^2(M))$. In addition, we introduce $B\in \Psi^0_{\delta} (M)\subset \B(\L^2(M))$ with principal symbol $[b]$, where $b\in S_{\delta}^0(M)$.
\begin{thm}\label{thm:semiclassfunccalcmfld}
The family of operators $\{B\circ\rho_h(P(h))\}_{h\in(0,1]}\subset \B(\L^2(M))$ has the following properties:
\begin{itemize}
\item
There exists a constant $h_0\in (0,1]$, a family of symbol functions $\{e_\alpha\}_{\alpha\in\mathcal{A}}\subset S_{\delta}(1_{\R^{2n}})$, and for each $h\in (0,h_0]$ an operator $R(h)\in \B(\L^2(M))$ such that
\bq
(B\circ \rho_h(P(h))(f)=\sum_{\alpha\in \mathcal{A}} \overline\varphi_\alpha\cdot \mathrm{Op}_h(e_\alpha)((\overline {\overline \varphi}_\alpha\cdot f)\circ\gamma_\alpha^{-1})\circ\gamma_\alpha + R(h)(f)\label{eq:firstassertion33}
\eq
holds for all $h\in (0,h_0]$ and all $f\in \L^2(M)$, and 
\[\norm{R(h)}_{\B(\L^2(M))}=\mathrm{O}\big(h^\infty\big)\qquad \text{as }h\to 0.\]
The operator $R(h)$ depends on $B$, $p$, $\rho_h$, and the choice of the functions $\{ \varphi_\alpha,\overline \varphi_\alpha,\overline{\overline \varphi}_\alpha\}_{\alpha \in \mathcal{A}}$.
\item For each $\alpha\in \mathcal{A}$, the symbol function $e_\alpha$ has an asymptotic expansion in $S_{\delta}(1_{\R^{2n}})$ of the form
\bq
e_\alpha \sim \sum_{j=0}^\infty e_{\alpha,j},\qquad e_{\alpha,j}\in S_{\delta}^{j(2\delta-1)}(1_{\R^{2n}}),\label{eq:secondassertion33}
\eq
where $e_{\alpha,j}$ is for fixed $h\in (0,h_0]$ an element of $\CT(\R^{2n})$, and 
\bq
e_{\alpha,0}=\big((\rho_h\circ p)\cdot b\cdot\varphi_\alpha\big)\circ(\gamma_\alpha^{-1},(\partial\gamma_\alpha^{-1})^T).\label{eq:notation32535}
\eq
Moreover, for each $\alpha,j$ and each fixed $h$ one has
\bq
\supp  e_{\alpha,j}\subset \supp   \big((\rho_h\circ p)\cdot b\cdot\varphi_\alpha\big)\circ(\gamma_\alpha^{-1},(\partial\gamma_\alpha^{-1})^T).\label{eq:wts344353}
\eq
\end{itemize} 
\end{thm}
\begin{proof}Let us first summarize briefly the strategy of the proof, which is divided into four steps. In Step $0$, we use the Helffer-Sj\"ostrand formula to reduce the calculations involving $\rho_h(P(h))$ to calculations involving the resolvent $(P(h)-z)^{-1}$ for $z\in \C$, $\text{Im }z\neq 0$, and estimates which are valid uniformly in $z$. In Step $1$, we construct a parametrix which approximates $(P(h)-z)^{-1}$ up to an explicitly given remainder operator. In order to construct the parametrix, we localize the problem using the finite atlas $\{(U_\alpha,\gamma_\alpha)\}_{\alpha\in\mathcal{A}}$ for $M$ and the partition of unity $\{\varphi_\alpha\}_{\alpha\in\mathcal{A}}$, obtaining a local parametrix for each coordinate chart, and sum up these local parametrices to a global parametrix. In Step $2$, we plug the result of Step $1$ into the Helffer-Sj\"ostrand formula which transforms the leading term in our calculations into a sum of pullbacks of operators in $\B(\L^2(\R^n))$. Then we can apply  the semiclassical functional calculus on $\R^n$, and in particular Theorem \ref{thm:semiclassfunccalcrn}. Finally, in Step $3$, we use the concrete form of the obtained symbol functions from Step $2$ to deduce the assertions (\ref{eq:firstassertion33}-\ref{eq:wts344353}). \\
{\bf Step 0.} The operator ${P}(h)-z$ is invertible in $\B(\L^2(M))$ for $z \in \C$, $\text{Im }z\neq 0$, see \cite[Lemma 14.6]{zworski}, and by the Helffer-Sj\"ostrand formula \cite[Theorem 8.1]{dimassi-sjoestrand} one has 
\[
\rho_h(P(h))=\frac 1{i\pi} \intop_{\C} \bar \partial_z \tilde \rho_h(z) (P(h)-z)^{-1} \d z,
\]
where $dz$ denotes the Lebesgue measure, $\tilde \rho_h:\C \to \C$ is the same almost analytic extension of $\rho_h$ as in (\ref{eq:extension}), and $\bar \partial _z =(\partial_x+i\partial_y)/2$ when $z=x+iy$. The Helffer-Sj\"ostrand formula shows that  $\rho_h \left(P(h)\right)$ will be expressed as a sum of pullbacks of operators in $ \B(\L^2(\R^n))$ once we establish the same for the resolvent $(P(h)-z)^{-1}$. This is our strategy. Now, by Lemma \ref{lem:symbfunc}, to the three global operators
\[
\breve P(h):\Cinft(M)\to \Cinft(M),\qquad P(h):\Sob^2(M)\to \L^2(M),\qquad (P(h)-z)^{-1}:\L^2(M)\to \Sob^2(M)
\]
there correspond three families of local operators, indexed by the finite atlas $\mathcal{A}$
\bqn
\breve P_\alpha(h):\CT(\R^n)\to \CT(\R^n),\qquad P_\alpha(h):\Sob_h^2(\R^n)\to \L^2(\R^n),\qquad (P_\alpha(h)-z)^{-1}:\L^2(\R^n)\to \Sob_h^2(\R^n),
\eqn
which are related to the global operators according to
\bq
\breve P(h)(f)=\breve{P}_{\alpha}(h)(f\circ \gamma_\alpha^{-1})\circ\gamma_\alpha\qquad\forall\; f\in \CT(M),\;\supp f\circ \gamma_\alpha^{-1}\subset K_\alpha.\label{eq:byconstree}
\eq
{\bf Step 1.}\quad In this step we will deduce a formula for $(P(h)-z)^{-1}$ using the family of local resolvents $\{(P_\alpha(h)-z)^{-1}\}_{\alpha\in\mathcal{A}}$. For each $\alpha\in \mathcal{A}$, denote by $\Phi_\alpha,\overline{\Phi}_{\alpha}, \overline{\overline \Phi}_{\alpha}$ the operators $\L^2(M)\to \L^2(M)$ given by pointwise multiplication with $\varphi_\alpha,\overline{\varphi}_{\alpha},\overline{\overline \varphi}_{\alpha}$, and by $\Psi_\alpha, \overline{\Psi}_{\alpha}, \overline{\overline \Psi}_{\alpha}$ the operators $\L^2(\R^n)\to \L^2(\R^n)$ given by pointwise multiplication with $\varphi_\alpha\circ \gamma_\alpha^{-1},\overline{\varphi}_{\alpha}\circ \gamma_\alpha^{-1},\overline{\overline \varphi}_{\alpha}\circ \gamma_\alpha^{-1}$, respectively. We will denote (bi-)restrictions of these operators to linear subspaces of their domains by the same symbols. Furthermore, let us introduce the pullback maps
\begin{align*}
{\gamma}_\alpha^*: \mathcal{L}(\CT(\R^n),\CT(\R^n))&\to \mathcal{L}(\CT(U_\alpha),\CT(U_\alpha)),\qquad A \mapsto \big(f\mapsto A(f\circ\gamma^{-1}_\alpha)\circ \gamma_\alpha\big),\\
{{\gamma}_\alpha^{-1}}^*: \mathcal{L}(\CT(U_\alpha),\CT(U_\alpha))&\to \mathcal{L}(\CT(\R^n), \CT(\R^n)),\qquad A \mapsto \big(f\mapsto A(f\circ\gamma_\alpha)\circ \gamma^{-1}_\alpha\big),
\end{align*}
which are each other's inverses. Elliptic regularity implies that the resolvent $(P_{\alpha}(h)-z)^{-1}$ induces an operator
\[
(P_{\alpha}(h)-z)^{-1}|_{\CT(\R^n)}:\CT(\R^n)\to \Cinft(\R^n).
\]
Regarding  $\CT(U_\alpha)$ as a subset of  $\Cinft(M)$ for each $\alpha$, we can define an operator $\Cinft(M)\to \Cinft(M)$ by
\[
Y(h,z):=\sum_{\alpha\in \mathcal{A}} {\gamma}_\alpha^*\big(\Psi_\alpha\circ(P_{\alpha}(h)-z)^{-1}|_{\CT(\R^n)}\big)\circ{\overline{\Phi}_\alpha}.
\]
Using that $\breve{P}(h)-z$ is a local operator and taking into account (\ref{eq:byconstree}), one now computes
\begin{align}
\nonumber Y(h,z)\circ (\breve{P}(h)-z)&=\sum_{\alpha\in\mathcal{A}} {\gamma}_\alpha^*\big(\Psi_\alpha\circ(P_{\alpha}(h)-z)^{-1}|_{\CT(\R^n)}\big)\circ{\overline{\Phi}_\alpha}\circ (\breve{P}(h)-z)\circ{\overline{\overline \Phi}}_{\alpha}\\
&=\sum_{\alpha\in\mathcal{A}}\Big[ {\gamma}_\alpha^*\big(\Psi_\alpha\circ(P_{\alpha}(h)-z)^{-1}|_{\CT(\R^n)}\big)\circ {\gamma}_\alpha^*\big(\breve{P}_\alpha(h)-z\big)\circ{\overline{\overline \Phi}}_{\alpha}  \label{eq:crucialline23535}\\
\nonumber &\qquad-\underbrace{{\gamma}_\alpha^*\big(\Psi_\alpha\circ(P_{\alpha}(h)-z)^{-1}|_{\CT(\R^n)}\big)\circ(\1-{\overline{\Phi}_\alpha})\circ (\breve{P}(h)-z)\circ{\overline{\overline \Phi}}_{\alpha}}_{=:\widetilde{\mathcal{R}}_\alpha(h,z)}\Big]\\
\nonumber &=\sum_{\alpha\in\mathcal{A}}\Big[  {\gamma}_\alpha^*\Big(\Psi_\alpha\circ(P_{\alpha}(h)-z)^{-1}|_{\CT(\R^n)}\circ (\breve{P}_\alpha(h)-z)\Big)\circ{\overline{\overline \Phi}}_{\alpha}- \widetilde{\mathcal{R}}_\alpha(h,z)\Big]\\
\nonumber &=\sum_{\alpha\in\mathcal{A}}\Big[  {\gamma}_\alpha^*\big(\Psi_\alpha\big)\circ{\overline{\overline \Phi}}_{\alpha} -  \widetilde{\mathcal{R}}_\alpha(h,z) \Big]
=\sum_{\alpha\in\mathcal{A}}\Big[ \Phi_{\alpha} -  \widetilde{\mathcal{R}}_\alpha(h,z)\Big] \\
 &=\1_{\Cinft(M)} -  \sum_{\alpha\in\mathcal{A}}\widetilde{\mathcal{R}}_\alpha(h,z).\label{eq:dshfjdkh}
\end{align}
Note how we inserted the additional cutoff operator ${\overline{\overline \Phi}}_{\alpha}$ before (\ref{eq:crucialline23535}) to be able to split off a remainder term which involves an operator that is composed from the left and from the right with multiplication operators by functions whose supports are disjoint. It immediately follows from (\ref{eq:dshfjdkh}) that
\begin{equation}
(P(h)-z)^{-1}|_{\Cinft(M)}=Y(h,z)+\sum_{\alpha\in\mathcal{A}}\mathcal{R}_\alpha(h,z),\label{eq:firstresolvform}
\end{equation}
where
\bqn
\mathcal{R}_\alpha(h,z):=\widetilde{\mathcal{R}}_\alpha(h,z)\circ (P(h)-z)^{-1}|_{\Cinft(M)}.
\eqn
We introduced the pullbacks $\gamma_{\alpha}^*$ and ${{\gamma}_\alpha^{-1}}^*$ for temporary use because they are inverses of each other and they respect compositions of operators, allowing the easy construction of the parametrix $Y(h,z)$ on $\Cinft(M)$. To get statements about operators in $\B(\L^2(M))$, we will work from now on with the pullback ${\Gamma}_\alpha^*$ from Lemma \ref{lem:pullbackfunctn}. Taking into account Lemma \ref{lem:symbfunc}, we observe that the bounded operator
\[
\sum_{\alpha\in \mathcal{A}}\Phi_\alpha\circ {\Gamma}_\alpha^*\big((P_{\alpha}(h)-z)^{-1}\big)\circ{\overline{\Phi}_\alpha}:\L^2(M)\to \L^2(M)
\]
agrees with $Y(h,z)$ on $\Cinft(M)$. As $\Cinft(M)$ is dense in $\L^2(M)$, it follows from (\ref{eq:firstresolvform}) that
\begin{equation}
(P(h)-z)^{-1}=\sum_{\alpha\in \mathcal{A}}\Big[\Phi_\alpha\circ {\Gamma}_\alpha^*\big((P_{\alpha}(h)-z)^{-1}\big)\circ{\overline{\Phi}_\alpha}+\mathfrak{R}_\alpha(h,z)\Big],\label{eq:desiredformula}
\end{equation}
where
\bq
\mathfrak{R}_\alpha(h,z):=\Phi_\alpha\circ {\Gamma}_\alpha^*\big((P_{\alpha}(h)-z)^{-1}\big)\circ(\1-{\overline \Phi}_\alpha)\circ(P(h)-z)\circ{\overline{\overline \Phi}}_{\alpha}\circ (P(h)-z)^{-1}.\label{eq:defoperator}
\eq
{\bf Step 2.}\quad  Plugging the result of Step $1$ into the Helffer-Sj\"ostrand formula yields
\begin{align*}
\rho_h(P(h))&=\frac 1{i\pi} \intop_{\C} \bar \partial_z \tilde \rho_h(z)\Big( \sum_{\alpha\in \mathcal{A}}\Big[\Phi_\alpha\circ {\Gamma}_\alpha^*\big((P_{\alpha}(h)-z)^{-1}\big)\circ{\overline{\Phi}_\alpha}+\mathfrak{R}_\alpha(h,z)\Big]\Big)\d z\\
&=\sum_{\alpha\in \mathcal{A}}\Big[\Phi_\alpha\circ {\Gamma}_\alpha^*\Big(\frac 1{i\pi} \intop_{\C} \bar \partial_z \tilde \rho_h(z) (P_{\alpha}(h)-z)^{-1} \d z\Big)\circ{\overline{\Phi}_\alpha} + \mathfrak{R}_\alpha(h)\Big],
\end{align*}
where
\[
\mathfrak{R}_\alpha(h)=\frac 1{i\pi} \intop_{\C} \bar \partial_z \tilde \rho_h(z)\mathfrak{R}_\alpha(h,z)\d z.
\]
For each $\alpha$, the functional calculus for the operator $P_{\alpha}(h)$ applies (by Lemma \ref{lem:symbfunc} and \cite[Theorem 8.1]{dimassi-sjoestrand}) and gives 
\[
\frac 1{i\pi} \intop_{\C} \bar \partial_z \tilde \rho_h(z)(P_{\alpha}(h)-z)^{-1}\d z=\rho_h\left(P_{\alpha}(h)\right).
\]
We obtain the result
\begin{equation}
\rho_h(P(h))= \sum_{\alpha\in \mathcal{A}}\Big[\Phi_\alpha\circ{\Gamma}_\alpha^*\big(\rho_h\left(P_{\alpha}(h)\right)\big)\circ{\overline{\Phi}_\alpha} + \mathfrak{R}_\alpha(h)\Big].\label{eq:resultstep1}
\end{equation}
This formula expresses the bounded operator $\rho_h(P(h)):\L^2(M)\to \L^2(M)$ in terms of the bounded operators $\rho_h\left(P_{\alpha}(h)\right):\L^2(\R^n)\to \L^2(\R^n)$, up to the remainder  $\sum_{\alpha\in \mathcal{A}}\mathfrak{R}_\alpha(h)$. We proceed by estimating for fixed $\alpha$ the operator norm of $\mathfrak{R}_\alpha(h)$. In order to do this, we note that with (\ref{eq:defoperator})
\[
\mathfrak{R}_\alpha(h)=\frac 1{i\pi} \intop_{\C} \bar \partial_z \tilde \rho_h(z)\Gamma_\alpha^*\Big(\Psi_\alpha\circ (P_{\alpha}(h)-z)^{-1}\circ(\1-{\overline \Psi}_\alpha)\Big)\circ(P(h)-z)\circ{\overline{\overline \Phi}}_{\alpha}\circ (P(h)-z)^{-1}\d z.
\]
Here we have replaced the operators $\Phi_\alpha$ and $(\1-{\overline \Phi}_\alpha)$ by the corresponding operators inside the pullback. We now want to estimate the operator norm of the remainder operator $\mathfrak{R}_\alpha(h)$ by estimating the operator norm of the integrand, which works because the integration domain is in fact the support of $\bar \partial_z \tilde \rho_h(z)$ which is compact and thus has finite volume. By Lemma \ref{lem:resolvestim} and Lemma \ref{lem:pullbackfunctn}, we get for each $N\in \N$  a constant $C_N>0$ such that for $|z|\leq$ const.\
\[
\norm{\Gamma_\alpha^*\Big(\Psi_\alpha\circ (P_{\alpha}(h)-z)^{-1}\circ(\1-{\overline \Psi}_\alpha)\Big)}_{\B(\L_{\mathrm{comp},\alpha}^2(M),\L^2(M))}\leq C_Nh^N |\text{Im }z|^{-N-1}.
\]
This crucial estimate is precisely the reason why it was helpful to split off the remainder term the way we did in (\ref{eq:crucialline23535}). Moreover, when introducing a commutator, we get for $|z|\leq $ const.\
\begin{multline*}
\norm{(P(h)-z)\circ{\overline{\overline \Phi}}_{\alpha}\circ (P(h)-z)^{-1}}_{\B(\L^2(M))}=\norm{\big[P(h),{\overline{\overline \Phi}}_{\alpha}\big]\circ (P(h)-z)^{-1}+{\overline{\overline \Phi}}_{\alpha}}_{\B(\L^2(M))}\\
\leq \norm{\big[P(h),{\overline{\overline \Phi}}_{\alpha}\big]\circ (P(h)-z)^{-1}}_{\B(\L^2(M))}+\norm{{\overline{\overline \Phi}}_{\alpha}}_{\B(\L^2(M))}\leq C h(|\text{Im }z|)^{-1}+1,
\end{multline*}
so that in total we obtain a new set of constants $\{C'_N\}$, $N=0,1,2,\ldots$, such that for $|z|\leq$ const.
\begin{multline}
\norm{\Gamma_\alpha^*\Big(\Psi_\alpha\circ (P_{\alpha}(h)-z)^{-1}\circ(\1-{\overline \Psi}_\alpha)\Big)\circ(P(h)-z)\circ{\overline{\overline \Phi}}_{\alpha}\circ (P(h)-z)^{-1}}_{\B(\L^2(M))}\\
\leq C'_Nh^N |\text{Im }z|^{-N-1}.\label{eq:operatornormestim45454}
\end{multline}
We thus have successfully estimated the operator norm of $\mathfrak{R}_\alpha(z,h)$. In order to estimate also the supremum norm of the function $\bar \partial _z  \tilde \rho_h(z)$, recall from (\ref{eq:betterestim4545}) that for $N=0,1,2,\ldots$, there is a constant $C_N>0$ such that
\bqn
|\bar \partial _z  \tilde \rho_h(z)| \leq C_N |\text{Im } z|^N\mathrm{vol}(\supp \rho_h) \max_{0\leq j \leq N+3}\norm{\rho_h^{(j)}}_{\infty}\qquad \forall\;z\in \C,\;\forall\; h\in (0,1].
\eqn
Together with (\ref{eq:operatornormestim45454}), this implies that we get for each $N\in \N $ a new constant $C_N>0$ such that for all $u\in \L^2(M)$ and for $|z|\leq$ const.
\bq
|\bar \partial_z \tilde \rho_h(z)|^2\norm{\mathfrak{R}_\alpha(z,h)u}_{\L^2(M)}^2
\leq C_N h^{2N}\mathrm{vol}(\supp \rho_h)^2 \max_{0\leq j \leq N+4}\norm{\rho_h^{(j)}}_{\infty}^2\norm{u}^2_{\L^2(M)}\quad\forall\; h\in (0,1].\label{eq:estim2353989}
\eq
Thus, for all $h\in(0,1]$ and $N\in\N$ one has
\begin{multline*}
\norm{\mathfrak{R}_\alpha(h)u}^2_{\L^2(M)}=\intop_M
\Big|\frac 1{i\pi} \intop_{\C} \bar \partial_z \tilde \rho_h(z)\mathfrak{R}_\alpha(z,h)(u)(x)\d z \Big|^2 \d M(x) \\
\leq \frac 1{\pi} \intop_{\C}\intop_M|\bar \partial_z \tilde \rho_h(z)|^2\big|\mathfrak{R}_\alpha(z,h)(u)(x) \big|^2 \d M(x)\d z
= \frac 1{\pi} \intop_{\C}|\bar \partial_z \tilde \rho_h(z)|^2\norm{\mathfrak{R}_\alpha(z,h)u}_{\L^2(M)}^2\d z\\
\leq C_N h^{2N}\mathrm{vol}_\C(\supp \bar\partial_z \tilde \rho_h)\mathrm{vol}(\supp \rho_h)^2 \max_{0\leq j \leq N+4}\norm{f_h^{(j)}}_{\infty}^2\norm{u}^2_{\L^2(M)}.
\end{multline*}
Note that (\ref{eq:restimates}) and (\ref{eq:betterestim4545}) imply for each $h\in (0,1]$ that the function $M\times \C\to \R$ given by $$(x,z)\mapsto |\bar \partial_z \tilde \rho_h(z)|^2\big|\mathfrak{R}_\alpha(z,h)(u)(x) \big|^2$$
has finite $\L^1$-norm with respect to the product measure $\d z \d M$. This justifies the application of the Fubini theorem. We are now in essentially the same situation as we were in (\ref{eq:intermediateresult1}), so that with analogous arguments as in the lines following (\ref{eq:intermediateresult1}) we conclude
\[
\norm{\mathfrak{R}_\alpha(h)u}^2_{\L^2(M)}=\mathrm{O}(h^\infty)\norm{u}^2_{\L^2(M)}
\]
with estimates independent of $u$, and as $u\in \L^2(M)$ was arbitrary, it follows $$\norm{\mathfrak{R}_\alpha(h)}_{\B(\L^2(M))}=\mathrm{O}(h^\infty).$$ The estimation of the operator norm of the remainder is now almost complete. Namely, since $\mathcal{A}$ is finite, we can re-write (\ref{eq:resultstep1}) as
\begin{equation}
\rho_h(P(h))= \sum_{\alpha\in \mathcal{A}}\Phi_\alpha\circ{\Gamma}_\alpha^*\big(\rho_h\left(P_{\alpha}(h)\right)\big)\circ{\overline{\Phi}_\alpha} + \widetilde{R}(h),\label{eq:resultstep2}
\end{equation}
where
\[
\widetilde{R}(h):=\sum_{\alpha\in \mathcal{A}}\mathfrak{R}_\alpha(h): \L^2(M)\to \L^2(M)
\]
has operator norm of order $h^\infty$. Next, we compose with the operator $B$ and an additional cutoff operator. That yields
\bq
B\circ \rho_h(P(h))= \sum_{\alpha\in \mathcal{A}}\overline\Phi_\alpha\circ B\circ\Phi_\alpha\circ{\Gamma}_\alpha^*\big(\rho_h\left(P_{\alpha}(h)\right)\big)\circ{\overline{\Phi}_\alpha} + R(h),\label{eq:operatordecomp628}
\eq
where
\bq
R(h):= \sum_{\alpha\in \mathcal{A}}(\1-\overline\Phi_\alpha)\circ B\circ\Phi_\alpha\circ{\Gamma}_\alpha^*\big(\rho_h\left(P_{\alpha}(h)\right)\big)\circ{\overline{\Phi}_\alpha} \;+\; B\circ\widetilde{R}(h).\label{eq:rhdef}
\eq
As $\varphi_\alpha$ and $1-\overline{\varphi}_\alpha$ have disjoint supports, the operator norm of $(\1-\overline\Phi_\alpha)\circ B\circ\Phi_\alpha$ is of order $h^\infty$. Moreover, by Lemma \ref{lem:pullbackfunctn} and the spectral theorem, we have 
\bqn
\norm{{\Gamma}_\alpha^*\big(\rho_h\left(P_{\alpha}(h)\right)\big)}_{\B(\L_{\mathrm{comp},\alpha}^2(M),\L^2(M))}
\leq C\norm{\rho_h\left(P_{\alpha}(h)\right)}_{\B(\L^2(\R^n))}\leq C\norm{\rho_h}_\infty \leq C'\qquad \forall \;h\in (0,1]
\eqn
with constants $C,C'>0$, the last inequality being a consequence of the assumption $\rho_h\in S_{\delta}(1_\R)$. In addition, we know that the operator norm of $\widetilde{R}(h)$ is of order $h^\infty$.  From these observations, it follows 
\[
\norm{R(h)}_{\B(\L^2(M))}=\mathrm{O}(h^\infty).
\]
{\bf Step 3.}\quad We now express the summands in the leading term of (\ref{eq:operatordecomp628}) as pullbacks of semiclassical pseudodifferential operators on $\R^n$. Since $B$ is a semiclassical pseudodifferential operator of order $(0,\delta)$ with principal symbol $[b]$, one has for each $\alpha$
\bq
{\overline{\Phi}_\alpha}\circ B\circ\Phi_\alpha={\overline{\Phi}_\alpha}\circ {\Gamma}_\alpha^*\big(\mathrm{Op}_h(b_\alpha)\big)\circ\Phi_\alpha,\qquad b_\alpha\in S_{\delta}(1_{\R^{2n}}),\label{eq:localBsymb}
\eq
with a symbol function $b_\alpha$ that has the property
\bq
b_\alpha=b\circ \gamma_\alpha^{-1} + h^{1-2\delta}\, \widetilde{b}_\alpha,\qquad  \widetilde{b}_\alpha\in S_{\delta}(1_{\R^{2n}}).\label{eq:localbsymbprop}
\eq
To proceed, note that by Lemma \ref{lem:symbfunc}, we can apply Theorem \ref{thm:semiclassfunccalcrn} to $P_{\alpha}(h)$ for each $\alpha$, which gives us a symbol function $s_\alpha\in \bigcap_{k\in \N}S_{\delta}(\mathfrak{m}^{-k})$, where $\mathfrak{m}(y,\eta)=\eklm{\eta}^2$, and a number $h_{0,\alpha}\in(0,1]$ such that for $h\in (0,h_{0,\alpha}]$
\bq
\rho_h\left(P_{\alpha}(h)\right)=\mathrm{Op}_h(s_\alpha).\label{eq:localsalpha}
\eq
Each local operator $\rho_h\left(P_{\alpha}(h)\right)$ is thus a semiclassical pseudodifferential operator. Moreover, Theorem \ref{thm:semiclassfunccalcrn} implies that there is an asymptotic expansion in $S_{\delta}(1/\mathfrak{m})$
\bq
s_\alpha\sim \sum_{j=0}^\infty s_{\alpha,j},\qquad 
s_{\alpha,j}(y,\eta,h)=\frac{1}{(2j)!}\Big(\frac{\partial}{\partial t}\Big)^{2j}\big(q_j(y,\eta,t,h)\rho_h(t)\big)_{t=p_{\alpha,0}(y,\eta)}\label{eq:expansion987654}
\eq
for a sequence of polynomials $\{q_j(t)\}_{j=0,1,2,\ldots}$ in one variable $t\in \R$ with coefficients being $h$-dependent functions in $\Cinft(\R^{2n})$ and with $q_0\equiv 1$. In particular, one has $s_{\alpha,j}\in S_{\delta}^{j(2\delta-1)}(1/\mathfrak{m})$ and 
\bq
s_{\alpha,0}(y,\eta,h)=\rho_h(p_{\alpha,0}(y,\eta)),\label{eq:princsymb667894128}
\eq
where 
\bq
p_{\alpha,0}(y,\eta)=\tau_\alpha(y)(|\eta|^2_{g_\alpha(y)}+V_\alpha(y))+(1-\tau_\alpha(y)) |\eta|^2\label{eq:defp3243890}
\eq
is the $h^0$-coefficient in the full symbol of $\breve P_\alpha(h)$, see (\ref{eq:symbfunclem}). Since $1/\mathfrak{m}\leq 1$, it holds $S_{\delta}^{j(2\delta-1)}(1/\mathfrak{m})\subset S_{\delta}^{j(2\delta-1)}(1_{\R^{2n}})$, so that we can replace in the statements above $S_{\delta}^{j(2\delta-1)}(1/\mathfrak{m})$ with $S_{\delta}^{j(2\delta-1)}(1_{\R^{2n}})$, obtaining in particular $s_\alpha \in S_{\delta}(1_{\R^{2n}})$. 
Set $h_0:=\min_{\alpha\in \mathcal{A}}h_{0,\alpha}>0$. By (\ref{eq:operatordecomp628}), (\ref{eq:localBsymb}), and (\ref{eq:localsalpha}), we have proved that one has for all $h\in (0,h_0]$
\bq
B\circ\rho_h(P(h))= \sum_{\alpha\in \mathcal{A}}\overline\Phi_\alpha\circ{\Gamma}_\alpha^*\big( \mathrm{Op}_h(b_\alpha)\circ \Psi_\alpha\circ \mathrm{Op}_h(s_\alpha)\big)\circ\overline{\Phi}_\alpha+R(h).
\eq
Let us now prove that the function $s_{\alpha,j}(y,\cdot,h): \eta \mapsto s_{\alpha,j}(y,\eta,h)$ is an element of $\CT(\R^n)$ for each $y,h,j$ which fulfills 
\bq
\supp  s_{\alpha,j}(\gamma_\alpha(x),\cdot,h)\subset (\partial\gamma_\alpha)^T\big((\supp  \rho_h\circ p)\cap T_x^*M\big)\quad \forall\;x\in \gamma_\alpha^{-1}(K_\alpha). 
\label{eq:estim6767676767}
\eq
Indeed, this statement follows from formula (\ref{eq:expansion987654}). By that formula, at each point $(y,\eta,h)\in \R^{2n}\times (0,1]$ the number $s_{\alpha,j}(y,\eta,h)$ is a polynomial in derivatives of $\rho_h$ at $p_{\alpha,0}(y,\eta)$. However, each derivative of $\rho_h$ has compact support inside $\supp \rho_h$, so that for each $y$, the function $\eta\mapsto s_{\alpha,j}(y,\eta,h)$ is supported inside $$\mathrm{supp}(\rho_h\circ p_{\alpha,0})\cap \{(y,\eta): \eta \in \R^n\}.$$
Since $\tau_\alpha\equiv 1$ on $K_\alpha$, it holds for $x\in \gamma_\alpha^{-1}(K_\alpha)$ $$\mathrm{supp}(\rho_h\circ p_{\alpha,0})\cap \{(\gamma_\alpha(x),\eta): \eta \in \R^n\}=(\partial\gamma_\alpha)^T\big((\supp  \rho_h\circ p)\cap T_x^*M\big).$$
This proves (\ref{eq:estim6767676767}). Now, we apply the composition formula to $\mathrm{Op}_h(b_\alpha)\circ \Psi_\alpha\circ \mathrm{Op}_h(s_\alpha)$, treating $\Psi_\alpha$ here as a zero order $h$-pseudodifferential operator. Theorem \ref{thm:compositionformula} then yields
\begin{align}
\nonumber \mathrm{Op}_h(b_\alpha)\circ \Psi_\alpha\circ \mathrm{Op}_h(s_\alpha)&=\mathrm{Op}_h(e_\alpha),\qquad e_\alpha\in S_{\delta}(1_{\R^{2n}}),\\
\nonumber e_\alpha&\sim \sum_{j=0}^\infty e_{\alpha,j},\qquad e_{\alpha,j}\in S_{\delta}^{j(2\delta-1)}(1_{\R^{2n}}),\\
e_{\alpha,0}&=\big((\varphi_\alpha\cdot b)\circ(\gamma_\alpha^{-1},(\partial\gamma_\alpha^{-1})^T)\big)\cdot (\rho_h\circ p_{\alpha,0}).\label{eq:expansion333366666}
\end{align}
Here we took (\ref{eq:localbsymbprop}) and (\ref{eq:princsymb667894128}) into account. 
The function $\varphi_\alpha$ is compactly supported inside $\gamma_\alpha^{-1}(K_\alpha)$, and we have seen that $s_\alpha$ has the expansion (\ref{eq:expansion987654}) in terms of symbol functions which are compactly supported in the co-tangent space variable $\eta$. The summands in the expansion (\ref {eq:composexpans}) of the composition formula are products of derivatives of the original symbol functions. Therefore, if one of the functions is compactly supported in the co-tangent space variable $\eta$, and the other one in the manifold variable $y$, the whole summand is compactly supported in $\R^{2n}$. Taking into account (\ref{eq:estim6767676767}), the statement (\ref{eq:wts344353}) follows. 
To finish the proof, we recall from (\ref{eq:defp3243890}) how $p_{\alpha,0}$ was defined, and that $\tau_\alpha \in \CT(\R^n)$ is identically $1$ on $K_\alpha$. Since $\varphi_\alpha\circ \gamma_\alpha^{-1}$ is supported inside $K_\alpha$, the claim (\ref{eq:notation32535}) finally follows. \end{proof}
\subsection{Trace norm estimates}
Our next goal is to deduce a refined version of Theorem \ref{thm:semiclassfunccalcmfld}, with a remainder operator of trace class. In order to achieve this, we need to relate the functional and symbolic calculi with trace norm remainder estimates. Suppose that we are in the situation introduced at the beginning of this section. For each $\alpha\in\mathcal{A}$, set 
\bq
u_{\alpha,0}:=\big((\rho_h\circ p)\cdot b\cdot \varphi_\alpha\big)\circ(\gamma_\alpha^{-1},(\partial\gamma_\alpha^{-1})^T),\label{eq:princsymbsymb}
\eq 
with $b$ as in Theorem \ref{thm:semiclassfunccalcmfld}. Then, one has the following result.
\begin{thm}\label{thm:localtraceform}Suppose that $\rho_h\in \mathcal{S}^{\mathrm{bcomp}}_{\delta}$.  Then, for each $N\in \N$, there is a number $h_0\in(0,1]$, a collection of symbol functions $\{r_{\alpha,\beta,N}\}_{\alpha,\beta\in \mathcal{A}}\subset S^{2\delta-1}_{\delta}(1_{\R^{2n}})$ and an operator $\mathfrak{R}_{N}(h)\in \B(\L^2(M))$ such that
\begin{itemize}
\item one has for all $f\in\L^2(M)$, $h\in (0,h_0]$ the relation \begin{multline*}
B\circ \rho_h(P(h))(f)=\sum_{\alpha \in \mathcal{A}}\overline \varphi_\alpha\cdot\mathrm{Op}_h(u_{\alpha,0})\big((f\cdot \overline{\overline {\overline \varphi}}_\alpha)\circ \gamma_\alpha^{-1}\big)\circ\gamma_\alpha\\
+\sum_{\alpha,\beta \in \mathcal{A}}\overline \varphi_\beta\cdot\mathrm{Op}_h(r_{\alpha,\beta,N})\big((f\cdot \overline{\overline \varphi}_\alpha\cdot\overline {\overline {\overline \varphi}}_\beta)\circ\gamma_\beta^{-1}\big)\circ\gamma_\beta  \;+\; \mathfrak{R}_{N}(h)(f);
\end{multline*}
\item the operator $\mathfrak{R}_N(h)\in \B(\L^2(M))$ is of trace class and its trace norm fulfills
\bq
\norm{\mathfrak{R}_N(h)}_{\mathrm{tr},\L^2(M)}=\mathrm{O}\big(h^{N}\big)\quad\text{as }h\to 0;\label{eq:tracenormestim343434}
\eq
\item for fixed $h\in (0,h_0]$, each symbol function $r_{\alpha,\beta,N}$ is an element of $\CT(\R^{2n})$ that fulfills 
\bq
\supp  r_{\alpha,\beta,N}\subset \supp   \big((\rho_h\circ p)\cdot b \cdot\varphi_\alpha\big)\circ(\gamma_\alpha^{-1},(\partial\gamma_\alpha^{-1})^T).\label{eq:wts798978}
\eq
\end{itemize}
\end{thm}
\begin{proof}The proof is divided into five steps. Let the notation be as in the proof of Theorem \ref{thm:semiclassfunccalcmfld}. \\
{\bf Step 0.}\quad Consider the collection of symbol functions $\{e_{\alpha,j}\}$ with $e_{\alpha,j}\in S_{\delta}^{j(2\delta-1)}(1_{\R^{2n}})$ obtained in the proof of Theorem \ref{thm:semiclassfunccalcmfld}. Let $R(h)\in \B(\L^2(M))$ be the remainder operator from (\ref{eq:rhdef}), whose operator norm is of order $h^\infty$. The statement (\ref{eq:secondassertion33}) means
$
e_{\alpha}- \sum_{j=0}^N e_{\alpha,j}\in S_{\delta}^{(N+1)(2\delta -1)}(1_{\R^{2n}}),
$
which by (\ref{eq:l2cont45}) implies $$\Big\Vert\mathrm{Op}_h(e_\alpha)-\sum_{j=0}^N\mathrm{Op}_h(e_{\alpha,j})\Big\Vert_{\B(\L^2(\R^n))} \leq C_\alpha h^{(1-2\delta)(N+1)},$$
with a constant $C_\alpha>0$ independent of $h$. Since $\mathcal{A}$ is finite, and applying analogous arguments as in the proof of Lemma \ref{lem:pullbackfunctn}, we obtain
\begin{multline*}
\Big\Vert\sum_{\alpha\in \mathcal{A}} \overline \varphi_\alpha \mathrm{Op}_h(e_\alpha)((\overline{\overline \varphi}_\alpha f)\circ\gamma_\alpha^{-1})\circ\gamma_\alpha-\sum_{\stackrel{\scriptstyle \alpha\in \mathcal{A}}{\scriptstyle 0 \leq j \leq N}} \overline \varphi_\alpha \mathrm{Op}_h(e_{\alpha,j})((\overline{\overline \varphi}_\alpha f)\circ\gamma_\alpha^{-1})\circ\gamma_\alpha\Big\Vert_{\B(\L^2(M))}\\ \leq C'h^{(1-2\delta)(N+1)}
\end{multline*}
for some constant $C'>0$ independent of $h$. Thus, setting
\[
R_N(h):=\sum_{\alpha\in \mathcal{A}} \overline\varphi_\alpha \mathrm{Op}_h(e_\alpha)((\overline{\overline \varphi}_\alpha f)\circ\gamma_\alpha^{-1})\circ\gamma_\alpha-\sum_{\stackrel{\scriptstyle \alpha\in \mathcal{A}}{\scriptstyle 0 \leq j \leq N}} \overline \varphi_\alpha \mathrm{Op}_h(e_{\alpha,j})((\overline{\overline \varphi}_\alpha f)\circ\gamma_\alpha^{-1})\circ\gamma_\alpha+R(h),
\]
we have
\bq
\norm{R_N(h)}_{\B(\L^2(M))}=\mathrm{O}(h^{(1-2\delta)(N+1)}),\label{eq:opnormestimate1}
\eq
and by Theorem \ref{thm:semiclassfunccalcmfld} we obtain for sufficiently small $h$ and each $f\in \L^2(M)$
\bq
B\circ \rho_h(P(h))(f)=\sum_{\stackrel{\scriptstyle \alpha\in \mathcal{A}}{\scriptstyle 0 \leq j \leq N}} \overline \varphi_\alpha \mathrm{Op}_h(e_{\alpha,j})((\overline{\overline \varphi}_\alpha f)\circ\gamma_\alpha^{-1})\circ\gamma_\alpha + R_N(h)(f). \label{eq:intermediateresult}
\eq
This looks promising, but since we are interested in \emph{trace norm} remainder estimates, there is still some work to do.\\
{\bf Step 1.}\quad To proceed, we recall that $P(h)$ has only finitely many eigenvalues $$E(h)_{1},\ldots, E(h)_{N(h)}$$ in $\supp \rho_h$, and the corresponding eigenspaces are all finite-dimensional. By the spectral theorem,  
\begin{equation}
\rho_h\left(P(h)\right)  =  \sum_{j=1}^{N(h)}\rho_h\left(E_{j}(h)\right)\Pi_{j}(h),\label{eq:spectral}
\end{equation}
where $\Pi_{j}(h)$ denotes the spectral projection onto the eigenspace of $P(h)$ corresponding to the eigenvalue $E_j(h)$. Hence, $\rho_h\left(P(h)\right)$ is a finite sum of projections onto finite-dimensional spaces and, consequently,  a finite rank operator, and therefore of trace class. \\
Now, we get prepared to use a trick that allows us to partially estimate trace norms by operator norms. The trick has been used already by Helffer and Robert in \cite[Proof of Prop.\ 5.3]{helffer-robert83}, and to implement it, we proceed as follows. For $h\in (0,1]$, choose $\overline{\rho}_h\in \CT(\R)$ such that
\bq
\overline{\rho}_h =1\qquad \text{on the support of }\rho_h,
\eq
and such that the function $(t,h)\mapsto \overline{\rho}_h(t)$ is an element of the symbol class\footnote{The larger symbol class $S_{\delta}(1_{\R})$ would also do. However, as the diameter of the support of $\overline{\rho}_h$ can be assumed to be bounded away from $0$, the symbol class $S(1_{\R})$ is more natural.} $S(1_{\R})$ and
\bq
\supp \overline{\rho}_h\subset \overline{I}\qquad\forall\; h\in (0,1]\label{eq:363636363633}
\eq
for some $h$-independent closed interval $\overline I\subset \R$. The abstract functional calculus given by the spectral theorem fulfills $f(A)\circ g(A)=(f\cdot g)(A)$ for any self-adjoint operator $A$ in a Hilbert space and any two bounded Borel functions $f,g$ on $\R$. We therefore get
\bq
{\rho}_h \left(P(h)\right)\circ\overline\rho_h \left(P(h)\right)=({\rho}_h \cdot\overline\rho_h) \left(P(h)\right)=\rho_h \left(P(h)\right)\quad\forall\;h\in(0,1]. \label{eq:specprod11}
\eq
Now, basic operator theory tells us that for operators $Z,S\in \B(\L^2(M))$ of which $Z$ is of trace class, $Z\circ S$ and $S\circ Z$ are also of trace class and it holds  
\begin{align}
\norm{Z\circ S}_{\mathrm{tr},\L^2(M)}&\leq \norm{Z}_{\mathrm{tr},\L^2(M)}\norm{S}_{\B(\L^2(M))},\label{eq:trickrels1}\\
\norm{S\circ Z}_{\mathrm{tr},\L^2(M)}&\leq \norm{Z}_{\mathrm{tr},\L^2(M)}\norm{S}_{\B(\L^2(M))}.\label{eq:trickrels2}
\end{align}
The trick is to use the latter estimates together with (\ref{eq:specprod11}) to estimate the trace norm of remainders by operator norms. Indeed, we can apply all our predecing results, and in particular Theorem \ref{thm:semiclassfunccalcmfld}, also to the operator $\overline{\rho}_h \left(P(h)\right)$. From now on, choose $h_0$ to be the minimum of the two $h_0$ we obtain for $\rho_h$ and $\overline{\rho}_h$ from Theorem \ref{thm:semiclassfunccalcmfld}. Choosing $B=\1_{\L^2(M)}$, one then has by  (\ref{eq:intermediateresult}) for $N=0,1,2,\ldots$, $f\in \L^2(M)$, and $h\in (0,h_0]$
\bq
\overline{\rho}_h(P(h))(f)=\sum_{\stackrel{\scriptstyle \alpha\in \mathcal{A}}{\scriptstyle 0 \leq j \leq N}}\overline \varphi_\alpha \mathrm{Op}_h(\overline{e}_{\alpha,j})((\overline{\overline \varphi}_\alpha f)\circ\gamma_\alpha^{-1})\circ\gamma_\alpha + \overline{R}_N(h)(f),\label{eq:intermediateresult22}
\eq
where $ \overline{R}_N(h)\in \B(\L^2(M))$ fulfills
\bq
\norm{\overline{R}_N(h)}_{\B(\L^2(M))}=\mathrm{O}(h^{N+1}),\label{eq:opnormestimate2}
\eq
and the symbols $\overline e_{\alpha,j}\in S^{-j}(1_{\R^{2n}})$ have analogous properties as the symbols $\{e_{\alpha,j}\}$. In particular,
\bq
\overline{e}_{\alpha,0}=\big((\overline{\rho}_h\circ p)\cdot \varphi_\alpha\big)\circ(\gamma_\alpha^{-1},(\partial\gamma_\alpha^{-1})^T)\qquad \forall\;h\in (0,h_0].\label{eq:princsymbsymb2}
\eq
We now use (\ref{eq:specprod11}) and (\ref{eq:intermediateresult22}) to get for $N=0,1,2,\ldots$ and $h\in (0,h_0]$:
\begin{align*}
& B\circ \rho_h(P(h)) =B\circ\rho_h(P(h))\circ{\overline \rho}_h(P(h))\\
&=\sum_{\stackrel{\scriptstyle \alpha\in \mathcal{A}}{\scriptstyle 0 \leq j \leq N}} B\circ{\rho}_h(P(h))\circ\overline\Phi_\alpha \circ \Gamma_\alpha^\ast\big(\mathrm{Op}_h(\overline e_{\alpha,j})\big)\circ\overline{\overline \Phi}_\alpha  +B\circ\rho_h(P(h))\circ  \overline R_N(h).
\end{align*}
From now on, we fix $N$ and assume $h_0$ to be small enough such that $\overline R_N(h)$ has operator norm less than $\frac{1}{2}$ for each $h\in (0,h_0]$, which implies that $\1_{\L^2(M)}-\overline R_N(h)$ is invertible. Note that this makes $h_0$ depend on $N$. Using also the corresponding Neumann series, one obtains for  $h\in (0,h_0]$ the equality
\begin{align*}
B\circ \rho_h(P(h)) &=\sum_{\stackrel{\scriptstyle \alpha\in \mathcal{A}}{\scriptstyle 0 \leq j \leq N}} B\circ{\rho}_h(P(h))\circ \Gamma_\alpha^\ast\big(\overline\Psi_\alpha \circ\mathrm{Op}_h(\overline e_{\alpha,j})\big)\circ\overline{\overline \Phi}_\alpha\circ\Big(\1_{\L^2(M)}-\overline R_N(h)\Big)^{-1} \\
&=\sum_{\stackrel{\scriptstyle \alpha\in \mathcal{A}}{\scriptstyle 0 \leq j \leq N}} B\circ{\rho}_h(P(h))\circ \Gamma_\alpha^\ast\big(\overline\Psi_\alpha \circ\mathrm{Op}_h(\overline e_{\alpha,j})\big)\circ\overline{\overline \Phi}_\alpha\circ\Big(\sum_{k=0}^\infty \overline R_N(h)^k\Big)\\
&=\sum_{\stackrel{\scriptstyle \alpha\in \mathcal{A}}{\scriptstyle 0 \leq j \leq N}} B\circ{\rho}_h(P(h))\circ \Gamma_\alpha^\ast\big(\overline\Psi_\alpha \circ\mathrm{Op}_h(\overline e_{\alpha,j})\big)\circ\overline{\overline \Phi}_\alpha\\
&+\overbrace{\sum_{\stackrel{\scriptstyle \alpha\in \mathcal{A}}{\scriptstyle 0 \leq j \leq N}} B\circ{\rho}_h(P(h))\circ \Gamma_\alpha^\ast\big(\overline\Psi_\alpha \circ\mathrm{Op}_h(\overline e_{\alpha,j})\big)\circ\overline{\overline \Phi}_\alpha\circ\Big(\sum_{k=1}^\infty \overline R_N(h)^k\Big)}^{:=\widetilde{\mathcal{R}}_N(h)}.
\end{align*}
To proceed, we insert (\ref{eq:intermediateresult}) into the first summand, which yields
\begin{align}
B\circ \rho_h(P(h)) &=\sum_{\stackrel{\scriptstyle \alpha,\beta \in \mathcal{A}}{\scriptstyle 0 \leq j,k \leq N}} \Gamma_\beta^\ast\big(\overline\Psi_\beta \circ \mathrm{Op}_h(e_{\beta,k})\big)\circ\overline{\overline \Phi}_\beta\circ \Gamma_\alpha^\ast\big(\overline\Psi_\alpha \circ\mathrm{Op}_h(\overline e_{\alpha,j})\big)\circ\overline{\overline \Phi}_\alpha\label{eq:decomp38239753892579}\\
\nonumber &\qquad + \overbrace{\sum_{\stackrel{\scriptstyle \alpha\in \mathcal{A}}{\scriptstyle 0 \leq j \leq N}} R_N(h)\circ \Gamma_\alpha^\ast\big(\overline\Psi_\alpha \circ\mathrm{Op}_h(\overline e_{\alpha,j})\big)\circ\overline{\overline \Phi}_\alpha+\widetilde{\mathcal{R}}_N(h)}^{:=\mathcal{R}_N(h)}.
\end{align}
We see that a drawback of the trick is that besides a new remainder term, we now also have a different leading term for each $N$. \\
{\bf Step 2.}\quad In this step we prove that the operator $\mathcal{R}_N(h)$ is in fact of trace class and satisfies a good trace norm estimate. Recall that the function $\overline\varphi_\alpha\circ\gamma_\alpha^{-1}$ is compactly supported inside the interior of the compactum $K_\alpha\subset \R^{n}$. Now, in view of (\ref{eq:estim6767676767}) and (\ref{eq:363636363633}), the results leading to (\ref{eq:symboltraceestims}) imply that $\overline\Psi_\alpha \circ \mathrm{Op}_h\big(\overline e_{\alpha,j}\big)$ is of trace class, and by (\ref{eq:symboltraceestims}) it holds for $h\in (0,h_0]$
\begin{multline*}
\norm{\overline\Psi_\alpha \circ \mathrm{Op}_h\big(\overline e_{\alpha,j}\big)}_{\mathrm{tr},\L^2(\R^n)}\\
\leq C_{\alpha,j} h^{-n}(1+ \mathrm{vol}_{\,T^*M}(\supp  \rho_h\circ p))\sum_{|\beta|\leq 2n+1}\max_{(y,\eta)\in K_\alpha\times \R^n}|\partial^\beta  \overline e_{\alpha,j}(\cdot,h)|
\end{multline*}
for some constant $C_{\alpha,j}>0$ which is independent of $h$. Next, we use that $\overline e_{\alpha,j}$ is an element of 
$S^{-j}(1_{\R^{2n}})$, which implies 
$$\sum_{|\beta|\leq 2n+1}\max_{(y,\eta)\in K_\alpha\times \R^n}|\partial^\beta  \overline  e_{\alpha,j}(\cdot,h)|\leq \widetilde C_{\alpha,n} h^{j}\qquad \forall\; h\in (0,h_0].$$
In summary, we obtain the estimate
\bq
\norm{\overline\Psi_\alpha \circ\mathrm{Op}_h\big(\overline e_{\alpha,j} \big)}_{\mathrm{tr},\L^2(\R^n)}\leq C'_{\alpha,j} (1+ \mathrm{vol}_{\,T^*M}(\supp  \rho_h\circ p))h^{j-n} \quad \forall\; h\in (0,h_0]\label{eq:firsttracenormestim33}
\eq
for some new constant $C'_{\alpha,j}>0$ which is independent of $h$. To proceed, note that as $M$ is compact and $\mathrm{Vol}_{g_\alpha}$ is bounded and on $K_\alpha$ also bounded away from zero, our trace norm estimates in $\L^2(\R^n)$ carry over to trace norm estimates in $\L^2(M)$ by using Schwartz kernel estimates similar to \cite[(9.1) on p.\ 112]{dimassi-sjoestrand}. Combining now (\ref{eq:trickrels1}),  (\ref{eq:trickrels2}), (\ref{eq:firsttracenormestim33}), and (\ref{eq:opnormestimate1}), we conclude
\begin{multline*}
\norm{R_N(h)\circ \Gamma_\alpha^\ast\big(\overline\Psi_\alpha \circ\mathrm{Op}_h(\overline e_{\alpha,j})\big)\circ\overline{\overline \Phi}_\alpha}_{\mathrm{tr},\L^2(M)}\\
\leq C_{\alpha,j} (1+ \mathrm{vol}_{\,T^*M}(\supp  \rho_h\circ p))h^{(N+1)(1-2\delta)+j-n} \qquad \forall\; h\in (0,h_0]
\end{multline*}
with a constant $C_{\alpha,j}>0$ that is independent of $h$, and it follows from the finiteness of $\mathcal{A}$ that there is $C>0$ such that
\begin{multline*}
\bigg\Vert \sum_{\stackrel{\scriptstyle \alpha\in \mathcal{A}}{\scriptstyle 0 \leq j \leq N}}R_N(h)\circ \Gamma_\alpha^\ast\big(\overline\Psi_\alpha \circ\mathrm{Op}_h(\overline e_{\alpha,j})\big)\circ\overline{\overline \Phi}_\alpha\bigg\Vert_{\mathrm{tr},\L^2(M)}\\
\leq C (1+ \mathrm{vol}_{\,T^*M}(\supp  \rho_h\circ p))h^{(N+1)(1-2\delta)-n} \qquad \forall\; h\in (0,h_0].
\end{multline*}
Similarly, taking into account that the operator norm of $\rho_h(P(h))$ is uniformly bounded in $h$ by the spectral theorem and the assumption that $\rho_h\in \mathcal{S}^\mathrm{bcomp}_\delta$, and writing
\[
\sum_{k=1}^\infty \overline R_N(h)^k=\overline R_N(h)\circ\Big(\sum_{k=0}^\infty \overline R_N(h)^k\Big)=\overline R_N(h)\circ \Big(\1_{\L^2(M)}-\overline R_N(h)\Big)^{-1},
\]
where $\big(\1_{\L^2(M)}-\overline R_N(h)\big)^{-1}$ has operator norm less than $2$ for $h\in (0,h_0]$, it follows from (\ref{eq:firsttracenormestim33}) and (\ref{eq:opnormestimate2}) that there is $C'>0$ such that
\bqn
\norm{\widetilde{\mathcal{R}}_N(h)}_{\mathrm{tr},\L^2(M)}\leq C' (1+ \mathrm{vol}_{\,T^*M}(\supp  \rho_h\circ p))h^{N+1-n} \qquad \forall\; h\in (0,h_0].
\eqn
By assumption, the diameter of the support of $\rho_h$ is bounded uniformly in $h$, and $M$ is compact, so the number $\mathrm{vol}_{\,T^*M}(\supp  \rho_h\circ p)$ is also bounded uniformly in $h$. From the last two estimates, we therefore conclude finally
\bq
\norm{\mathcal{R}_N(h)}_{\mathrm{tr},\L^2(M)}\leq C h^{(N+1)(1-2\delta)-n} \qquad \forall\; h\in (0,h_0]\label{eq:goodtracenormestim11}
\eq
with a new constant $C>0$ that is independent of $h$. Our estimation of the trace norm of the remainder operator $\mathcal{R}_N(h)$ is finished.\\
{\bf Step 3.}\quad We now turn our attention to the leading term in (\ref{eq:decomp38239753892579}) with summands given by
\[
\Gamma_\beta^\ast\big(\overline\Psi_\beta \circ \mathrm{Op}_h(e_{\beta,k})\big)\circ\overline{\overline \Phi}_\beta\circ \Gamma_\alpha^\ast\big(\overline\Psi_\alpha \circ\mathrm{Op}_h(\overline e_{\alpha,j})\big)\circ\overline{\overline \Phi}_\alpha.
\]
The problem with these terms is that they do not yet have the right form as claimed in the first statement of Theorem \ref{thm:localtraceform}, in particular they involve two pullbacks, one along the chart $\gamma_\alpha$ and one along $\gamma_\beta$, and we need to combine them into a single pullback. This will be done using a coordinate transformation from the $\alpha$-th chart to the $\beta$-th chart. Before we can perform this transformation, we need to localize further to the intersection of both chart domains. To this end, note that since $\overline {\overline \varphi}_\beta$ and $1-\overline{\overline {\overline \varphi}}_\beta$ have disjoint supports, we have
\[
\overline{\overline \Phi}_\beta\circ \Gamma_\alpha^\ast\big(\overline\Psi_\alpha \circ\mathrm{Op}_h(\overline e_{\alpha,j})\big)\circ\overline{\overline \Phi}_\alpha= \overline{\overline \Phi}_\beta\circ \Gamma_\alpha^\ast\big(\overline\Psi_\alpha \circ\mathrm{Op}_h(\overline e_{\alpha,j})\big)\circ\overline{\overline \Phi}_\alpha\circ\overline{\overline {\overline \Phi}}_\beta + R_{\alpha,\beta,j}(h),
\]
for a remainder operator $R_{\alpha,\beta,j}(h)\in \B(\L^2(M))$ with $\norm{R_{\alpha,\beta,j}(h)}_{\B(\L^2(M))}=\mathrm{O}(h^\infty)$. Similarly as in (\ref{eq:firsttracenormestim33}), it follows that the operator $\Gamma_\beta^*\big(\overline\Psi_\beta\circ \mathrm{Op}_h(e_{\beta,k})\big)$ is of trace class in $\L^2(M)$, and its trace norm is of order $(1+\mathrm{vol}_{\,T^*M}(\supp  \rho_h\circ p))h^{j(1-2\delta)-(2n+1)\delta-n}$. Again, the number $\mathrm{vol}_{\,T^*M}(\supp  \rho_h\circ p)$ is  bounded uniformly in $h$. Thus, the trace norm of $\Gamma_\beta^*\big(\overline \Psi_\beta\circ \mathrm{Op}_h({e}_{\beta,k})\big)$ is bounded uniformly in $h$. Therefore, setting
\[
\mathcal{R}_{\alpha,\beta,j,k}(h):=\Gamma_\beta^*\big(\overline \Psi_\beta\circ \mathrm{Op}_h({e}_{\beta,k})\big)\circ R_{\alpha,\beta,j}(h)
\]we conclude
\bq
\norm{\mathcal{R}_{\alpha,\beta,j,k}(h)}_{\mathrm{tr},\L^2(M)}\leq \norm{\Gamma_\beta^*\big(\overline \Psi_\beta\circ \mathrm{Op}_h({e}_{\beta,k})\big)}_{\mathrm{tr},\L^2(M)}\norm{R_{\alpha,\beta,j}(h)}_{\B(\L^2(M))}=\mathrm{O}(h^\infty).\label{eq:goodtracenormestim22}
\eq
The reason why we inserted the cutoff operator $\overline{\overline {\overline \Phi}}_\beta$ corresponding to the function $\overline{\overline {\overline \varphi}}_\beta$ is that we are now prepared to perform the required coordinate transformation. Indeed, one has
\[
(\overline{\overline \varphi}_\alpha \overline {\overline{\overline \varphi}}_\beta f)\circ\gamma_\alpha^{-1}=(\overline{\overline \varphi}_\alpha \overline{ \overline{\overline \varphi}}_\beta f)\circ\gamma_\beta^{-1}\circ\gamma_\beta\circ\gamma_\alpha^{-1},
\]
which leads to
\begin{multline}
\overline\varphi_\beta \mathrm{Op}_h({e}_{\beta,k})\Big(\Big(\overline{\overline \varphi}_\beta \overline{\varphi}_\alpha \mathrm{Op}_h(\overline e_{\alpha,j})((\overline{\overline \varphi}_\alpha f)\circ\gamma_\alpha^{-1})\circ\gamma_\alpha\Big)\circ\gamma_\beta^{-1}\Big)\circ\gamma_\beta\\
=\overline\varphi_\beta \mathrm{Op}_h({e}_{\beta,k})\Big(\big((\overline{\overline \varphi}_\beta \overline \varphi_\alpha)\circ\gamma_\beta^{-1} \big)\big[\mathrm{Op}_h(\overline e_{\alpha,j})((\overline{\overline\varphi}_\alpha \overline{\overline {\overline \varphi}}_\beta f)\circ\gamma_\beta^{-1}\circ\Theta_{\alpha\beta}^{-1})\big]\circ\Theta_{\alpha\beta}\Big)\circ\gamma_\beta\\ + \mathcal{R}_{\alpha,\beta,j,k}(h)(f),\label{eq:opmaps}
\end{multline}
where $\mathcal{R}_{\alpha,\beta,j,k}(h)$ fulfills $\norm{\mathcal{R}_{\alpha,\beta,j,k}(h)}_{\mathrm{tr},\L^2(M)}=\mathrm{O}(h^\infty)$, as shown above, and we introduced $$\Theta_{\alpha\beta}:=\gamma_\alpha\circ\gamma_\beta^{-1}:\gamma_\beta(U_\alpha\cap U_\beta)\to \gamma_\alpha(U_\alpha\cap U_\beta).$$
By the coordinate-transformation formula \cite[Theorem 9.3]{zworski}, it then holds
\[
\big((\overline{\overline \varphi}_\beta \overline \varphi_\alpha)\circ\gamma_\beta^{-1} \big)\big[\mathrm{Op}_h(\overline e_{\alpha,j})((\overline{\overline\varphi}_\alpha \overline{\overline {\overline \varphi}}_\beta f)\circ\gamma_\beta^{-1}\circ\Theta_{\alpha\beta}^{-1})\big]\circ\Theta_{\alpha\beta} = \mathrm{Op}_h(\overline u_{\alpha,\beta,j})((\overline{\overline\varphi}_\alpha \overline{\overline {\overline \varphi}}_\beta f)\circ\gamma_\beta^{-1})
\]
for a new symbol function $\overline u_{\alpha,\beta,j}\in S^{-j}(1_{\R^{2n}})$ which is for each fixed $h$ a Schwartz function on $\R^{2n}$ and fulfills
\bq\overline u_{\alpha,\beta,j}(y,\eta,h)=
(\overline{\overline \varphi}_\beta \overline \varphi_\alpha)\circ\gamma_\beta^{-1}(y) \underbrace{\overline e_{\alpha,j}(\Theta_{\alpha\beta}(y),\partial\Theta_{\alpha\beta}(\Theta_{\alpha\beta}(y))^T \eta,h)}_{=:\Theta_{\alpha\beta}^\ast\overline e_{\alpha,j}(y,\eta,h)}\; + h \;\overline r_{\alpha,\beta,j}(y,\eta,h),\label{eq:localuexpr}
\eq
with a remainder symbol function $\overline r_{\alpha,\beta,j} \in S^{-j}(1_{\R^{2n}})$ that is for each fixed $h$ a Schwartz function on $\R^{2n}$, too. We thus obtain for $f\in \L^2(M)$ the equality
\begin{multline}
\Gamma_\beta^*\big(\overline \Psi_\beta\circ \mathrm{Op}_h({e}_{\beta,k})\big)\Big(\overline{\overline { \varphi}}_\beta \overline \varphi_\alpha \mathrm{Op}_h(\overline e_{\alpha,j})((\overline{\overline \varphi}_\alpha f)\circ\gamma_\alpha^{-1})\circ\gamma_\alpha\Big)\\
= \Gamma_\beta^*\Big(\overline \Psi_\beta\circ \mathrm{Op}_h({e}_{\beta,k})\circ \mathrm{Op}_h(\overline u_{\alpha,\beta,j}) \circ \overline {\overline \Psi}_{\alpha\beta} \Big)\circ \overline {\overline {\overline \Phi}}_\beta(f),\label{eq:intermediate2424}
\end{multline}
where $\overline {\overline \Psi}_{\alpha\beta}:\L^2(\R^n)\to \L^2(\R^n)$ is the operator given by pointwise multiplication with the function $\overline {\overline \varphi}_\alpha\circ \gamma_\beta^{-1}$. Finally, we can apply the composition formula, described in Theorem \ref{thm:compositionformula}. It tells us that
\bq
\mathrm{Op}_h({e}_{\beta,k})\circ \mathrm{Op}_h(\overline u_{\alpha,\beta,j})=\mathrm{Op}_h(s_{\alpha,\beta,j,k}),\label{eq:operatorpresent}
\eq
where $s_{\alpha,\beta,j,k}\in S^{(j+k)(2\delta -1)}_{\delta}(1_{\R^{2n}})$. Moreover, Theorem \ref{thm:compositionformula} says that $s_{\alpha,\beta,j,k}$ has an asymptotic expansion in $S^{(j+k)(2\delta -1)}_{\delta}(1_{\R^{2n}})$:
\bq
s_{\alpha,\beta,j,k}\sim \sum_{l=0}^\infty s_{\alpha,\beta,j,k,l},\qquad s_{\alpha,\beta,j,k,l}\in S^{(j+k+l)(2\delta -1)}_{\delta}(1_{\R^{2n}}),\label{eq:expansion4789457839467}
\eq
where 
\bq
s_{\alpha,\beta,j,k,0}={e}_{\beta,k}\cdot ((\overline{\overline \varphi}_\beta \cdot\overline \varphi_\alpha)\circ\gamma_\beta^{-1})\cdot \Theta_{\alpha\beta}^\ast\overline e_{\alpha,j}.\label{eq:productsymb5527}
\eq
Similarly as in our first application of the composition formula after (\ref{eq:expansion333366666}), we conclude from the relations
\begin{align*}
\supp e_{\beta,k}&\subset \supp   \big((\rho_h\circ p)\cdot b \cdot\varphi_\beta\big)\circ(\gamma_\beta^{-1},(\partial\beta_\beta^{-1})^T),\\
 \supp \overline{e}_{\alpha,j}&\subset \supp   \big((\overline \rho_h\circ p)\cdot b \cdot\varphi_\alpha\big)\circ(\gamma_\alpha^{-1},(\partial\gamma_\alpha^{-1})^T)
\end{align*}
that $s_{\alpha,\beta,j,k,l}$ is compactly supported inside $$\supp  \Big(\big((\rho_h\circ p)\cdot b \cdot\varphi_\beta\big)\circ(\gamma_\beta^{-1},(\partial\beta_\beta^{-1})^T)\Big)\cap  \supp  \Big(\big((\overline \rho_h\circ p)\cdot b\cdot\varphi_\alpha\big)\circ(\gamma_\alpha^{-1},(\partial\gamma_\alpha^{-1})^T)\Big)\subset \R^{2n}$$ for each $l$ and each fixed $h\in (0,h_0]$, and consequently its support fulfills
\[
 \text{vol supp }s_{\alpha,\beta,j,k,l}\leq C_{\alpha,\beta,j,k,l} \mathrm{vol}_{\,T^*M}(\supp  \rho_h\circ p)\qquad \forall\; h\in (0,h_0]
\]
with some constant $C_{\alpha,\beta,j,k,l}>0$ that is independent of $h$. It also follows that  $s_{\alpha,\beta,j,k}$ is for each fixed $h\in (0,h_0]$ a Schwartz function on $\R^{2n}$. By (\ref{eq:expansion4789457839467}), we have for each $M\in \N$ that
\[
\mathfrak{R}_{\alpha,\beta,j,k,M}:=s_{\alpha,\beta,j,k}-\sum_{l=0}^M s_{\alpha,\beta,j,k,l}\in S^{(j+k+M+1)(2\delta -1)}_{\delta}(1_{\R^{2n}}),
\]
and Corollary \ref{cor:supercor} says that $\mathrm{Op}_h(\mathfrak{R}_{\alpha,\beta,j,k,M})$ is of trace class, with a trace norm bound for $h\in (0,h_0]$
\bq
\norm{\mathrm{Op}_h(\mathfrak{R}_{\alpha,\beta,j,k,M})}_{\mathrm{tr},\L^2(\R^n)}
\leq C_{\alpha,\beta,j,k,M}h^{(j+k+M+1)(1-2\delta)-(2n+1)\delta-n},\label{eq:goodtracenormestim33}
\eq
where $C_{\alpha,\beta,j,k,M}>0$ is independent of $h$. The fact that we need Corollary \ref{cor:supercor} here, which requires the considered symbol functions to be supported inside an $h$-independent compactum in $\R^{2n}$, is the \emph{only} reason why we need the additional assumption in this theorem that $\rho_h\in \mathcal{S}^\mathrm{bcomp}_\delta$. Collecting everything together, we get from (\ref{eq:decomp38239753892579}-\ref{eq:goodtracenormestim33})  for each $N,M\in \N$:
\begin{multline*}
B\circ \rho_h(P(h))\\
=\sum_{\stackrel{\scriptstyle \alpha,\beta \in \mathcal{A}}{\scriptstyle 0 \leq j,k \leq N}}\bigg[\Gamma_\beta^*\Big(\overline \Psi_\beta\circ \mathrm{Op}_h\Big(\sum_{\scriptstyle 0 \leq l \leq M} s_{\alpha,\beta,j,k,l}+\mathfrak{R}_{\alpha,\beta,j,k,M}\Big)\circ \overline {\overline \Psi}_{\alpha\beta} \Big)\circ \overline {\overline {\overline \Phi}}_\beta
+\mathcal{R}_{\alpha,\beta,j,k}(h)\bigg]+\mathcal{R}_{N}(h).
\end{multline*}
This is the final result of Step 3. We have transformed the leading term of (\ref{eq:decomp38239753892579}) into a more desired form that involves only pullbacks by one chart at a time. \\
{\bf Step 4.}\quad We complete the proof by setting $$\quad u_{\alpha,\beta,0}:=s_{\alpha,\beta,0,0,0},\quad u_{\alpha,\beta,M,N}:=\sum_{\stackrel{\scriptstyle 0 \leq j,k \leq N}{\scriptstyle 0 \leq l \leq M}}s_{\alpha,\beta,j,k,l}-s_{\alpha,\beta,0,0,0},$$
\bqn
\mathfrak{R}_{M,N}(h):=\sum_{\stackrel{\scriptstyle \alpha,\beta \in \mathcal{A}}{\scriptstyle 0 \leq j,k \leq N}}\Big[\Gamma_\beta^*\big(\overline \Psi_\beta\circ\mathrm{Op}_h(\mathfrak{R}_{\alpha,\beta,j,k,M})\circ \overline {\overline \Psi}_{\alpha\beta}\big)\circ \overline {\overline {\overline \Phi}}_\beta+\mathcal{R}_{\alpha,\beta,j,k}(h)\Big]+\mathcal{R}_{N}(h).
\eqn
One then has for each $M,N\in \N$
\bqn
B\circ \rho_h(P(h))=\sum_{\alpha,\beta \in \mathcal{A}}\Gamma_\beta^*\Big(\overline \Psi_\beta\circ\mathrm{Op}_h\big(u_{\alpha,\beta,0} + u_{\alpha,\beta,M,N}\big)\circ \overline {\overline \Psi}_{\alpha\beta}\Big)\circ \overline {\overline {\overline \Phi}}_\beta + \mathfrak{R}_{M,N}(h),
\eqn
where $u_{\alpha,\beta,0}\in S_{\delta}(1_{\R^{2n}})$ and $u_{\alpha,\beta,M,N}\in S^{2\delta -1}_{\delta}(1_{\R^{2n}})$ are elements of $\CT(K_\alpha\cap K_\beta\times \R^n)\subset \CT(\R^{2n})$ for each fixed $h\in (0,h_0]$, and
\bqn
\norm{\mathfrak{R}_{M,N}(h)}_{\mathrm{tr},\L^2(M)}=\mathrm{O}\Big(h^{(\min(N,M)+1)(1-2\delta)-n-(2n+1)\delta} \Big)\quad \text{as }h\to 0.
\eqn
Let $\widetilde{N}\in \N$. Since $1-2\delta>0$, we can find numbers $N(\widetilde{N}),M(\widetilde{N})\in \N$ large enough such that
\bqn
\norm{\mathfrak{R}_{M(\widetilde{N}),N(\widetilde{N})}(h)}_{\mathrm{tr},\L^2(M)}=\mathrm{O}\big(h^{\widetilde{N}}\big)\quad \text{as }h\to 0.
\eqn
Defining
\[
\mathfrak{R}_{\widetilde{N}}(h):=\mathfrak{R}_{M(\widetilde{N}),N(\widetilde{N})}(h),\qquad r_{\alpha,\beta,\widetilde{N}}:=u_{\alpha,\beta,M(\widetilde{N}),N(\widetilde{N})}\in S^{2\delta -1}_{\delta}(1_{\R^{2n}}),
\]
we arrive for arbitrary $\widetilde{N}\in \N$ at the equality
\bqn
B\circ \rho_h(P(h))=\sum_{\alpha,\beta \in \mathcal{A}}\Gamma_\beta^*\Big(\overline \Psi_\beta\circ\mathrm{Op}_h(u_{\alpha,\beta,0}+r_{\alpha,\beta,\widetilde{N}})\circ \overline {\overline \Psi}_{\alpha\beta}\Big)\circ \overline {\overline {\overline \Phi}}_\beta + \mathfrak{R}_{\widetilde{N}}(h).
\eqn
To finish the proof, recall the identities
\bqn
e_{\beta,0}=\big((\rho_h\circ p)\cdot \varphi_\beta\big)\circ(\gamma_\beta^{-1},(\partial\gamma_\beta^{-1})^T),\qquad \overline{e}_{\alpha,0}=\big((\overline{\rho}_h\circ p)\cdot \varphi_\alpha\big)\circ(\gamma_\alpha^{-1},(\partial\gamma_\alpha^{-1})^T).
\eqn
With these identities and the definition of the pullback by the function $\Theta_{\alpha\beta}\equiv \gamma_\alpha\circ\gamma_\beta^{-1}$, one computes
\begin{align*}
&{e}_{\beta,0}\cdot \Theta_{\alpha\beta}^\ast\overline e_{\alpha,0}
={e}_{\beta,0}\cdot\Big(\overline e_{\alpha,0}\circ (\Theta_{\alpha\beta},\partial \Theta_{\alpha\beta}^T)\Big)\\
&=\Big(\big((\rho_h\circ p)\cdot \varphi_\beta\big)\circ\big(\gamma_\beta^{-1},(\partial\gamma_\beta^{-1})^T\big)\Big)
\cdot\Big(\big((\overline{\rho}_h\circ p)\cdot \varphi_\alpha\big)\circ(\gamma_\alpha^{-1},(\partial\gamma_\alpha^{-1})^T)\circ (\Theta_{\alpha\beta},\partial \Theta_{\alpha\beta}^T)\Big)\\
&=\Big(\big((\rho_h\circ p)\cdot \varphi_\beta\big)\circ\big(\gamma_\beta^{-1},(\partial\gamma_\beta^{-1})^T\big)\Big)
\\&\qquad\qquad\cdot\Big(\big((\overline{\rho}_h\circ p)\cdot \varphi_\alpha\big)\circ(\gamma_\alpha^{-1},(\partial\gamma_\alpha^{-1})^T)\circ \big(\gamma_\alpha\circ\gamma_\beta^{-1},(\partial\gamma_\alpha)^T\circ(\partial\gamma_\beta^{-1})^T\big)\Big)\\
&=\big((\rho_h\circ p)\cdot(\overline \rho_h\circ p)\cdot \varphi_\alpha\cdot \varphi_\beta\big)\circ\big(\gamma_\beta^{-1},(\partial\gamma_\beta^{-1})^T\big).
\end{align*}
Taking finally into account that the functions decorated with a bar are identically $1$ on the supports of the corresponding functions without bar, it holds
\bqn
u_{\alpha,\beta,0}\equiv s_{\alpha,\beta,0,0,0}={e}_{\beta,0}\cdot ((\overline{\overline \varphi}_\beta \cdot\overline \varphi_\alpha)\circ\gamma_\beta^{-1})\cdot \Theta_{\alpha\beta}^\ast\overline e_{\alpha,0}
=\big((\rho_h\circ p)\cdot \varphi_\alpha\cdot \varphi_\beta\big)\circ\big(\gamma_\beta^{-1},(\partial\gamma_\beta^{-1})^T\big).
\eqn
In particular, since $\sum_{\alpha\in\mathcal{A}}\varphi_\alpha =1_M$, we can set $$u_{\beta,0}:=\sum_{\alpha\in\mathcal{A}}u_{\alpha,\beta,0}=\big((\rho_h\circ p)\cdot \varphi_\beta\big)\circ\big(\gamma_\beta^{-1},(\partial\gamma_\beta^{-1})^T\big)$$
which finally yields
\[
\sum_{\alpha,\beta \in \mathcal{A}}\Gamma_\beta^*\Big(\overline \Psi_\beta\circ\mathrm{Op}_h(u_{\alpha,\beta,0})\circ \overline {\overline \Psi}_{\alpha\beta}\Big)\circ \overline {\overline {\overline \Phi}}_\beta=\sum_{\beta \in \mathcal{A}}\Gamma_\beta^*\Big(\overline \Psi_\beta\circ\mathrm{Op}_h(u_{\beta,0})\Big)\circ \overline {\overline {\overline \Phi}}_\beta.
\]
\end{proof}
From the previous theorem one immediately deduces
\begin{cor}[Semiclassical trace formula for Schr\"odinger operators]\label{thm:semiclassicaltraceformula}In the situation of the previous theorem, one has in the semiclassical limit $h\to 0$
\begin{multline}
\mathrm{tr}\,_{\L^2(M)}\big[B\circ \rho_h(P(h))\big]\\
=\frac{1}{(2\pi h)^n}\intop_{T^*M}b\cdot(\rho_h\circ p)\d(T^*M)\;+\;\mathrm{O}\Big(h^{1-2\delta-n}\mathrm{vol}_{\,T^*M}\big[\supp \big(b\cdot (\rho_h\circ p)\big)\big]\Big).\label{eq:result474747433}
\end{multline}
\end{cor}
\begin{rem}
If $\mathrm{vol}_{\,T^*M}\big[\supp \big(b\cdot (\rho_h\circ p)\big)\big]\neq 0$, i.e.\ in all non-trivial cases, we can divide both sides of (\ref{eq:result474747433}) by $\mathrm{vol}_{\,T^*M}\big[\supp \big(b\cdot (\rho_h\circ p)\big)\big]$ to obtain the equivalent statement
\[
(2\pi h)^n\frac{\mathrm{tr}\,_{\L^2(M)}\big[B\circ\rho_h(P(h))\big]}{\mathrm{vol}_{\,T^*M}\big[\supp \big(b\cdot (\rho_h\circ p)\big)\big]}\\
=\fintop_{\text{supp }b\cdot(\rho_h\circ p)}b\cdot(\rho_h\circ p)\d(T^*M)\;+\;\mathrm{O}\big(h^{1-2\delta}\big)\qquad\text{as }h\to 0
\]
in which the distinction between the leading term and the remainder term is emphasized more.
\end{rem}
\begin{proof} For convenience of the reader, we give the short proof which involves only standard arguments. By Theorem \ref{thm:localtraceform}, there is a number $h_0\in(0,1]$ and for each $N\in \N$ a collection of symbol functions $\{r_{\alpha,\beta,N}\}_{\alpha,\beta\in \mathcal{A}}\subset S^{2\delta-1}_{\delta}(1_{\R^{2n}})$ and an operator $\mathfrak{R}_{N}(h)\in \B(\L^2(M))$  such that for $h\in (0,h_0]$
\begin{multline*}
B\circ \rho_h(P(h))(f)=\sum_{\alpha \in \mathcal{A}}\overline \varphi_\alpha\cdot\mathrm{Op}_h(u_{\alpha,0})\big((f\cdot \overline{\overline {\overline \varphi}}_\alpha)\circ \gamma_\alpha^{-1}\big)\circ\gamma_\alpha\\
+\sum_{\alpha,\beta \in \mathcal{A}}\overline \varphi_\beta\cdot\mathrm{Op}_h(r_{\alpha,\beta,N})\big((f\cdot \overline{\overline \varphi}_\alpha\cdot\overline {\overline {\overline \varphi}}_\beta)\circ\gamma_\beta^{-1}\big)\circ\gamma_\beta  + \mathfrak{R}_{N}(h)(f)\qquad\forall\; f\in \L^2(M),
\end{multline*}
where $u_{\alpha,0}=\big((\rho_h\circ p)\cdot b\cdot \varphi_\alpha\big)\circ(\gamma_\alpha^{-1},(\partial\gamma_\alpha^{-1})^T)$. Moreover, the operator $\mathfrak{R}_N(h)\in \B(\L^2(M))$ is of trace class and its trace norm is of order $h^{N}$ as $h\to 0$, while for fixed $h\in (0,h_0]$, each symbol function $r_{\alpha,\beta,N}$ is an element of $\CT(\R^{2n})$ that fulfills 
\bq
\text{vol supp } r_{\alpha,\beta,N}\leq C_{\alpha,\beta,N}\, \text{vol supp }\big((\rho_h\circ p)\cdot b\big)\label{eq:newrel243366}
\eq
with a constant $C_{\alpha,\beta,N}>0$ that is independent of $h$. In particular, each of the operators $$A_\alpha:f\mapsto \overline \varphi_\alpha\cdot\mathrm{Op}_h(u_{\alpha,0})\big((f\cdot \overline{\overline {\overline \varphi}}_\alpha)\circ \gamma_\alpha^{-1}\big), \quad A_{\alpha,\beta,N}:f\mapsto \overline \varphi_\beta\cdot\mathrm{Op}_h(r_{\alpha,\beta,N})\big((f\cdot \overline{\overline \varphi}_\alpha\cdot\overline {\overline {\overline \varphi}}_\beta)\circ\gamma_\beta^{-1}\big)\circ\gamma_\beta$$ 
has a smooth, compactly supported Schwartz kernel, given by
\begin{align*}
K_{A_\alpha}(x_1,x_2)&=\frac{1}{(2\pi h)^n} \overline \varphi_\alpha(x_1)\intop_{\R^{n}}e^{\frac i h (y_1-y_2)\cdot\eta}
u_{\alpha,0}\Big(\frac{y_1+y_2}{2},\eta,h\Big)\overline{\overline {\overline \varphi}}_\alpha(\gamma_\alpha^{-1}(y_2))  \d \eta\big(\mathrm{Vol}_{g_\alpha}(y)\big)^{-1}\\
K_{A_{\alpha,\beta,N}}(x_1,x_2)&=\frac{1}{(2\pi h)^n} \overline \varphi_\beta(x_1)\intop_{\R^{n}}e^{\frac i h (y_1-y_2)\cdot\eta}
r_{\alpha,\beta,N}\Big(\frac{y_1+y_2}{2},\eta,h\Big)(\overline{\overline \varphi}_\alpha\cdot\overline {\overline {\overline \varphi}}_\beta)(\gamma_\beta^{-1}(y_2))  \d \eta\big(\mathrm{Vol}_{g_\beta}(y)\big)^{-1},
\end{align*}
where $x_1,x_2\in M$ and $y_i$ denotes $\gamma_\alpha(x_i)$ and $\gamma_\beta(x_i)$ in the first line and the second line, respectively.  We obtain for arbitrary $N\in \N$
\bqn
\mathrm{tr}\,_{\L^2(M)}\big[B\circ \rho_h(P(h))\big] = \sum_{\alpha \in \mathcal{A}}\intop_M K_{A_\alpha}(x,x)\d M(x)
+\sum_{\alpha,\beta \in \mathcal{A}}\intop_M K_{A_{\alpha,\beta,N}}(x,x)\d M(x)   + \mathrm{O}(h^N).
\eqn
Let us consider first the integrals in the second summand. Using (\ref{eq:newrel243366}) we obtain that there is a constant $C_{\alpha,\beta,N}>0$, independent of $h$, such that
\bqn
\bigg|\intop_M K_{A_{\alpha,\beta,N}}(x,x)\d M(x)\bigg| \leq C_{\alpha,\beta,N}\frac{1}{(2\pi h)^n} \norm{\overline \varphi_\beta}_\infty\norm{r_{\alpha,\beta,N}}_\infty \norm{\overline{\overline \varphi}_\alpha\cdot\overline {\overline {\overline \varphi}}_\beta}_\infty \text{vol supp }\big((\rho_h\circ p)\cdot b\big).
\eqn
As $r_{\alpha,\beta,N}$ is an element of $S_{\delta}^{2\delta-1}(1_{\R^{2n}})$, one has $\norm{r_{\alpha,\beta,N}}_\infty=\mathrm{O}(h^{1-2\delta})$, and so we conclude
\[
\bigg|\intop_M K_{A_{\alpha,\beta,N}}(x,x)\d M(x)\bigg| =\mathrm{O}\Big(h^{1-2\delta-n} \text{vol supp }\big((\rho_h\circ p)\cdot b\big)\Big)\quad \text{as }h\to 0.
\]
Since $\mathcal{A}$ is finite, it follows
\begin{align*}
\mathrm{tr}\,_{\L^2(M)}\big[B\circ \rho_h(P(h)) \big] = \sum_{\alpha \in \mathcal{A}}\intop_M K_{A_\alpha}(x,x)\d M(x)
+ \mathrm{O}\big(h^{1-2\delta-n}\text{vol supp }((\rho_h\circ p)\cdot b\big).
\end{align*}
To finish the proof, we calculate the leading term to be given by
\begin{align*}
\sum_{\alpha \in \mathcal{A}}&\intop_M K_{A_\alpha}(x,x)\d M(x)=\sum_{\alpha \in \mathcal{A}}\frac{1}{(2\pi h)^n} \intop_{\R^{2n}}\overline \varphi_\alpha(\gamma_\alpha^{-1}(y))u_{\alpha,0}(y,\eta,h)\overline{\overline {\overline \varphi}}_\alpha(\gamma_\alpha^{-1}(y))   \d \eta\d y\\
&=\frac{1}{(2\pi h)^n}\sum_{\alpha \in \mathcal{A}}\; \intop_{\R^{2n}}\overline \varphi_\alpha(\gamma_\alpha^{-1}(y))\big((\rho_h\circ p)\cdot b\cdot \varphi_\alpha\big)(\gamma_\alpha^{-1}(y),(\partial\gamma_\alpha^{-1})^T\eta,h)\overline{\overline {\overline \varphi}}_\alpha(\gamma_\alpha^{-1}(y))   \d \eta\d y\\
&=\frac{1}{(2\pi h)^n}\sum_{\alpha \in \mathcal{A}}\; \intop_{T^*M}((\rho_h\circ p)\cdot b)(x,\xi)\cdot \varphi_\alpha(x)\d (T^*M)(x,\xi)\\
&=\frac{1}{(2\pi h)^n}\intop_{T^*M}((\rho_h\circ p)\cdot b)(x,\xi)\d (T^*M)(x,\xi).
\end{align*}
\end{proof}

\bibliography{bibliography}
\bibliographystyle{amsplain}

\end{document}